\documentclass[final]{siamltex}

\usepackage{amssymb,amsmath}
\usepackage{amsfonts}
\usepackage{graphicx}
\usepackage{color}

\def\R{\mathbb{R}}
\newtheorem{remark}[theorem]{\indent {Remark}}

\providecommand{\com[2]}{\begin{tt}[#1: #2]\end{tt}}

\providecommand{\prt}[1]{\left( #1 \right)}

\providecommand{\abs[1]}{\left|#1\right|}
\providecommand{\floor[1]}{\lfloor  #1 \rfloor }

 \providecommand{\quotes[1]}{``#1"}
\providecommand{\aor}[1]{{#1}}
\providecommand{\jmhr}[1]{#1}
\providecommand{\aorev}[1]{{#1}}
\providecommand{\jmh}[1]{#1}
\providecommand{\ao}[1]{#1}
\usepackage{url}

\title{On symmetric continuum opinion dynamics
}

\author{Julien M. Hendrickx\thanks{ICTEAM Institute, Universit\'{e} catholique de Louvain,
Belgium, 
({\tt julien.hendrickx@uclouvain .be}). His work is supported by the Belgian Network DYSCO (Dynamical Systems, Control, and Optimization), funded by the Interuniversity Attraction Poles Program, initiated by the Belgian Science Policy Office.}
        \and Alex Olshevsky\thanks{Department of ISE, University of Illinois at Urbana-Champaign, ({\tt aolshev2@illinois.edu}).}
}

\begin{document}

\maketitle

\begin{abstract} This paper investigates the asymptotic behavior of some common opinion dynamic models 
in a continuum of agents. We show that as long as the interactions among the agents are 
symmetric, the distribution of the agents' opinion converges. We 
also investigate whether convergence occurs in a stronger sense than merely in distribution, namely, whether the opinion of almost
every agent converges. We show that while this is not the case in general, it becomes true under  
plausible assumptions on inter-agent interactions, namely that agents with similar opinions exert
a non-negligible pull on each other, \jmh{or that the interactions are entirely determined by their opinions via a smooth function}.
\footnotetext{\aor{A preliminary version of this paper appeared in the Proceedings of the IEEE CDC 2013.}}
\end{abstract}

\begin{keywords} multiagent systems, 
opinion dynamics, consensus, Lyapunov stability.
\end{keywords}

\begin{AMS} 93D20, 91C20, 93A14.
\end{AMS}

\pagestyle{myheadings}
\thispagestyle{plain}

\section{Introduction}

There has been much recent interest within the control community
in the study of multi-agent systems in which the agents interact according to simple,
local rules, resulting in coordinated global behavior. Unfortunately, the dynamics
describing the interactions of such systems are often time-varying and  
nonlinear and their analysis appears to be at present impossible without making considerably simplifying 
assumptions. For instance, it is a common assumption in much of the literature on multi-agent control that the graphs 
governing the inter-agent interactions satisfy some sort of long-term connectivity condition. This 
assumption is necessary due to the apparent intractability of analyzing the long-term connectivity
properties associated with the graph of a multi-agent system governed by time-varying and nonlinear
local interactions.

Encouraging results \aor{without} long-term connectivity conditions have however been obtained for opinion dynamics models, one of the simplest and most natural class of multi-agent systems
with time-varying interactions. These models have been recently
proposed (see \cite{H-K}, \cite{K00}, and the survey \cite{L2007}) to model opinion changes resulting
from repeated personal interactions among individuals and have attracted considerable attention within the 
control community (see \cite{L2005, bht09, bht10, cft12,cf12,tf11}).  Much of this attention is due to the 
similarities between opinion dynamics models and various \aor{dynamics arising in multi-agent control.} \aor{Indeed,} opinion dynamics are \aor{also} nonlinear \aor{when} the inter-agent interactions  change with time and depend on on the states of the agents. 
Consequently, it is believed that the techniques developed to rigorously analyze the asymptotic
behavior of these models will be useful in the analysis of more complex multi-agent systems whose nonlinearity
arises from state-dependent and time-varying inter-agent interactions.

It has recently been shown for opinion dynamics systems with finitely many agents that the symmetry of the inter-agent interactions (or actually a weak symmetry condition called \quotes{cut-balance}) was sufficient to guarantee the convergence of all agents, \emph{independently of any long-term connectivity condition} \cite{HT:2013}, something which makes the analysis of the asymptotic behavior of numerous systems considerably easier. 
Related observations for discrete-time systems were made in \cite{BHOT:2005,Moreau:2005,L2005} (see also references in \cite{HT:2013} \aor{as well as \cite{dimos} and \cite{chaz-dyn}}).

Such general results are lacking for systems involving infinitely many agents, or a continuous mass of agents, even though partial results have been obtained under specific assumptions on the way interactions take place \cite{cft12,bht09} (the authors of \cite{tf11} also consider systems involving a continuous mass of agents, but focuse on the existence and uniqueness of solutions, and on the possibility of approximating it by finite-dimensional systems).
Our goal in this paper is to analyze the extent to which the results obtained for finitely many agents \cite{HT:2013} or for some specific models \cite{bht09,cft12} remain valid for general opinion dynamics models.

\subsection{Model description}

We now give a precise statement of the dynamics we will study.
We consider the functions $x_t(\alpha): \aor{[0,\infty) \times} [0,1] \rightarrow \R$ which are solutions of the equation 
\begin{equation} 
\label{mainiteration} 
\dot{x}_t(\alpha) = \int_{0}^{1} w \left( t, \alpha, \beta, x_t(\alpha), x_t(\beta) \right) ( x_t(\beta) - x_t (\alpha)) ~d \beta, 
\end{equation} 
with initial condition $x_0(\alpha) = x_0$. Here $w(\cdot, \cdot, \cdot, \cdot, \cdot)$ is a nonnegative function. The equation has a natural interpretation: each agent $\alpha \aor{\in [0,1]}$ continually adjusts its 
``opinion'' $x_t(\alpha)$ to move closer to the opinions of other agents, giving to each agent $\beta$ a nonnegative weight $w \left( t, \alpha, \beta, x_t(\alpha), x_t(\beta) \right) $, which  may depend on the identities of the two agents, their opinions, and on time.

Often, the function $w(\cdot, \cdot, \cdot, \cdot, \cdot)$ is taken to be $1_{|x_t(\alpha) - x_t(\beta)| < r}$ for some ``opinion radius $r$,'' which corresponds to every agent adjusting its opinion based only on the opinions of other like-minded agents; \aor{this is the so-called Hegselmann-Krause model}. We will not be making this assumption here and will instead study the more
general case. Thus the weight agent $\alpha$ accords agent $\beta$ can vary based on the time $t$, the indices $\alpha$ and $\beta$, and the values $x_t(\alpha), x_t(\beta)$\footnote{We note that it is possible to omit the dependence on $x_t(\alpha)$ and $x_t(\beta)$ without loss of generality by appropriately modifying $w(\cdot, \cdot, \cdot, \cdot, \cdot)$. Nevertheless, we will continue writing $w(\cdot, \cdot, \cdot, \cdot, \cdot)$ as a function of five arguments  for simplicity of presentation.}. 

Several technical assumptions are necessary for Eq. (\ref{mainiteration}) to make sense. The function $w(\cdot, \cdot, \cdot, \cdot, \cdot)$ is assumed to be \aor{jointly measurable in all  variables}. A constraint on $w(\cdot, \cdot, \cdot, \cdot, \cdot)$ is needed to ensure that the integral in Eq. (\ref{mainiteration}) is finite; we will assume that $w(\cdot, \cdot, \cdot, \cdot, \cdot)$ is bounded, i.e., there exists some constant $W\aor{<\infty}$ so that $w(t, \alpha, \beta, x_t(\alpha), x_t(\beta)) \leq W$ for all values of $t, \alpha, \beta, x_t(\alpha), x_t(\beta)$. \aor{Since our focus in this paper is on properties of solutions of Eq. (\ref{mainiteration}) when they exist, we will not \jmhr{analyze in detail the} questions of existence and uniqueness; rather}  we will
assume \aor{in the body of the paper} that \ao{for each $\alpha \in [0,1]$, $x_t(\alpha)$ is 
an absolutely continuous 
function such that} Eq. (\ref{mainiteration}) is satisfied\aor{\footnote{\aor{Note that all of our main results on properties of solutions thus hold for any solution of Eq. (\ref{mainiteration}) that exists. Furthermore, we prove in Appendix A that 
existence and uniqueness of the kind of solutions we study does actually hold under some Lipschitz assumptions on the function $w(\cdot, \cdot, \cdot, \cdot, \cdot)$.} }} for \ao{almost all $t$}. We will assume that the initial distribution of opinions is bounded, i.e., $x_0(\alpha) \in [0,1]$ for all $\alpha \in [0,1]$, and that solutions do not explode in finite time, in the sense that $\sup_{\alpha\in [0,1],t \leq T} \aor{|}x_t(\alpha)\aor{|}$ is finite for every $T\geq 0$ (We will \aor{later} see that all opinions then  remain in $[0,1]$ for all\footnote{\aor{In essence, we are unable to rule out the possibility that Eq. (\ref{mainiteration}) may have certain ``pathological'' solutions which are not absolutely continuous or explode in finite time. Thus we restrict our attention in this paper to study of properties of any solutions that are well-behaved, i.e., absolutely continuous with bounded suprema. In general, whether solutions which satisfy these properties exist, or whether all solutions satisfy these properties, are open questions. However, in Appendix A we prove that subject to some continuity and Lipschitz assumptions on the function $w(\cdot, \cdot, \cdot, \cdot, \cdot)$, solutions which satisfy all the properties we have assumed here do indeed exist.}} $t$.)
 {\em We will be making these assumptions throughout the remainder of this paper \aor{(except for Appendix A on existence and uniqueness of solutions)} without mention.}

\subsection{\jmhr{Particular cases of \eqref{mainiteration}}  }

\jmhr{We now present some instantiations of Eq. \eqref{mainiteration}, and show that our general model not only encompasses many classical models having appeared in the literature, but also allows representing new classes of more complex models.}

\bigskip

\aor{\noindent {\em Example 1. Consensus of finitely many agents.} Continuous consensus dynamics have been widely studied 
in multi-agent control (see e.g., \cite{moreau, olfati, ren, cheb, HT:2013}). These are the dynamics given by
\begin{equation} \label{cons} \dot{z}_i(t) = \sum_{j=1}^n w_{ij}(t) (z_j(t)-z_i(t)), ~~~ i = 1, \ldots, n,\end{equation} where $w_{ij}(t)$ are arbitrary nonnegative numbers.  Intuitively, there are $n$ distinct agents which run dynamics that repeatedly push the values $z_i$ closer to each other with ``interaction strengths'' $w_{ij}(t)$. Under some connectivity
and regularity conditions, one can show that each $z_i(t)$ will \jmhr{converge} to the same number \cite{moreau}, which is why these dynamics are called ``consensus dynamics.'' }

\aor{Continuous consensus dynamics may be viewed as a special case of Eq. (\ref{mainiteration}). Indeed, define $\Gamma_i = [(i-1)/n, i/n]$, and for any $\alpha, \beta \in [0,1]$, if $\alpha \in \Gamma_i$ and $\beta \in \Gamma_j$, set $w(\alpha, \beta, x_t(\beta), x_t(\alpha) = nw_{ij}(t)$. 
Then it is immediate that if $\alpha \in \Gamma_i$ then $x_t(\alpha) = z_i(t)$. In other words, this particular choices of $w(\cdot, \cdot, \cdot, \cdot, \cdot)$ leads the agents to replicate continuous consensus dynamics. Note that if $w_{ij}(t) = w_{ji}(t)$, the function $w(\cdot, \cdot, \cdot, \cdot, \cdot)$ is symmetric in the sense of being unaffected by the interchange of $\alpha$ and $\beta$.}

\aor{Consensus dynamics are a common tool throughout multi-agent control. We mention their use in coverage control \cite{gcb08}, formation control \cite{ofm07, othesis},
distributed estimation \cite{XBL05, XBL06}, distributed task assignment \cite{CBH09},  and distributed optimization \cite{TBA86} and \cite{NO09}. It is often the case that multi-agent controllers are designed either by a direct reduction to a consensus task or by using
consensus dynamics (in either continuous or discrete time) as a subroutine. }

\aor{Due to the widespread use of consensus dynamics, it is of considerable interest to analyze their behavior in large-scale systems,  i.e., in the setting when the number of agents goes to infinity. \jmhr{Some} initial progress on questions of this type was 
made in \cite{giacomo},} \jmhr{and we believe that the analysis of \eqref{mainiteration} could prove instrumental in understanding the asymptotic properties of such systems.}

\bigskip

\aor{\noindent {\em Example 2. The Hegselmann-Krause model.}  The model of opinion dynamics introduces by Hegselmann and Krause \cite{H-K} corresponds to
\[ w(t, \alpha, \beta, x_t(\alpha), x_t(\beta)) = {{\bf 1}}_{|x_t(\beta)-x_t(\alpha)| < r}. \] Here $x_t(\alpha)$ is usually thought of as 
an opinion of agent $\alpha$. Intuitively, each agent ignores all opinions which deviate from its own by more than a certain ``opinion radius'' $r$. However, each agent does move its opinion towards the average of those other agents whose opinions are at most $r$ from its own. Note that Hegselmann and Krause studied in \cite{H-K} a discrete-time version of this model.}

\aor{The original paper by Hegselmann and Krause \cite{H-K} spawned a vast follow-up literature analyzing variations of the model, some of which we
discuss next. It is, in general, not possible to relate here all the observations that have been made about this model. Nonetheless, we'd like to take this opportunity to mention a series of recent 
breakthroughs analyzing the convergence time of the (asymmetric) version of this model in discrete-time \cite{rasoul, chaz, wedin}.}

\bigskip

\aor{\noindent {\em Example 3. Hegselmann-Krause models with unequal radii.} It was proposed in the recent paper \cite{ana} to study the unequal radii version,
\[ w(t, \alpha, \beta, x_t(\alpha), x_t(\beta)) = {{\bf 1}}_{|x_t(\beta)-x_t(\alpha)| < r_{\alpha}}. \] In contrast to the standard Hegselmann-Krause model, here each agent $\alpha$ has a potentially different 
opinion radius $r_{\alpha}$ which determines its openness to opinions of others. In the terminology of \cite{ana}, this is a {\em synchronized
bounded confidence} model. It was also proposed in \cite{ana} to study 
\[ w(t, \alpha, \beta, x_t(\alpha), x_t(\beta)) = {{\bf 1}}_{|x_t(\beta)-x_t(\alpha)| < r_{\beta}}. \] This corresponds to each agent 
``broadcasting'' its opinion to others, with different agents having different ``loudness'' and consequently reaching more agents. In the language of \cite{ana},
this is a {\em synchronized bounded influence model.} We remark that \cite{ana} studied these models in discrete, rather than continuous time (see also the related \cite{ana2}). Moreover, we note that unlike the Hegselmann-Krause models, here $w(\cdot, \cdot, \cdot, \cdot, \cdot)$ may not be symmetric in the sense that the weight agent $\alpha$ places on $\beta$ may be different than the weight agent $\beta$ places on $\alpha$.} The results presented in this paper do thus not always apply to such systems.


\bigskip

\aor{\noindent {\em Example 4. Some other opinion dynamics.} The generality of Eq. (\ref{mainiteration}) includes many other models of
opinion dynamics. For example, it is natural to replace the sharp cutoff of the Hegselmann-Krause model with a smoother decline. This
corresponds to agents which take the opinions of all other agents into account, but with decreasing strength as these opinions deviate from
theirs. Such models have appeared in the literature; for example \cite{cft-ifac} considered weight functions which are continuous functions of $x_t(\beta)-x_t(\alpha)$ which drop to zero outside of a certain radius. Another possibility is to model the decline of opinion influence with a Gaussian decay, e.g., as \begin{equation} w(t, \alpha, \beta, x_t(\alpha), x_t(\beta)) = e^{-\frac{(x_t(\alpha) - x_t(\beta))^2}{\sigma_a^2}}. \label{gaussdecay} \end{equation}}
In other models  motivated by robotic applications where sensors cannot accurately sense robots that are too close (see e.g. \cite{mahen2014}), weights are positive only when agents are separated by a distance within a certain range that does not include 0,
e.g.  $w(t, \alpha, \beta, x_t(\alpha), x_t(\beta)) = 1$ if $|(x_t(\alpha) - x_t(\beta)| \in [R_{\min},R_{\max}]$ and 0 else.

\aor{Alternatively, Eq. (\ref{mainiteration}) is general enough so that the usual assumption of homogeneity in opinion dynamic models may be dispensed with:
for example, we may instead assume that agents come in $k$ different {\em types} and $w(t, \alpha, \beta, x_t(\alpha), x_t(\beta))$ depends not only
on the distance between $x_t(\alpha)$ and $x_t(\beta)$ but also on the types of $\alpha$ and $\beta$. This would allow us to model a scenario wherein agents put more weight on the opinion of similar agents. We may also consider a model along these lines with a continuous measure of similarity, e.g., 
\begin{equation} w(t, \alpha, \beta, x_t(\beta), x_t(\alpha)) = {\bf 1}_{|\alpha - \beta| < r'} ~ \cdot ~{\bf 1}_{|x_t(\alpha) - x_t(\beta)| < r}. \label{cont-type} \end{equation} This corresponds to agents which ignore the opinions not only of agents whose opinions are distant from theirs, but also agents which are not sufficiently similar. Note that Eq. (\ref{cont-type}) features a $w(\cdot, \cdot, \cdot, \cdot, \cdot)$ which is symmetric in the sense of being unaffected by interchange of $\alpha$ and $\beta$; on the other hand, 
the model from Eq. (\ref{gaussdecay}) described above will not be symmetric unless all the parameters $\sigma_{\alpha}$ are identical. }

\aor{Finally, we note that since $w(t, \alpha, \beta, x_t(\alpha), x_t(\beta))$ and $w(t, \alpha', \beta, x_t(\alpha'), x_t(\beta))$ can be arbitrarily different functions whenever $\alpha \neq \alpha'$, Eq. (\ref{mainiteration}) can also model any mixture of the models we have just described. }

\bigskip

\aor{As the previous examples make clear, Eq. (\ref{mainiteration}) captures the properties of a number of popular opinion dynamic models, and furthermore, the generality of Eq. (\ref{mainiteration}) allows us to describe a wide multitude of models not yet considered in the literature. It is therefore interesting to investigate how solutions of Eq. (\ref{mainiteration}) behave. We next turn to our results on this subject.}

\subsection{Main results} \aor{Unfortunately,} in general Eq. (\ref{mainiteration}) is not guaranteed to converge in any meaningful sense. Indeed, \aor{as Example 2 above makes clear, consensus of finitely many agents}
is a special case of Eq. (\ref{mainiteration}) and consequently all the counterexamples to convergence of finitely many agents from, \aor{for example,} \cite{parallelbook} can be adapted to this setting. 

\aor{We briefly spell out one example. Consider the dynamic system which switches at integer times between the following two systems:
\[ \dot{z}_1 = 0, \dot{z}_2  =  z_1 - z_2, 
\dot{z}_3  =  0,  \phantom{aaaa}\text{     and     }\phantom{aaaa}
 \dot{z}_1 =  0, \dot{z}_2  =  z_3 - z_2, \dot{z}_3  =  0 .\] 
This may be viewed as an instance of continuous consensus dynamics of Example 1 in Section 1.2 by 
defining the coefficients $w_{ij}(t)$ appropriately. It is easy to see that if $z_1(0)  < z_2(0)  < z_3(0)$ then this dynamic leads to $z_2(t)$ diverging as it alternates between being pulled in the direction of $z_1(0)$ and being ``pulled'' in the direction of $z_3(0)$. }

\aor{As we discussed in Example 1 of Section 1.2, continuous consensus dynamics are a special case of Eq. (\ref{mainiteration}). We therefore conclude that,
in general, it is not true that if $x_t(\alpha)$ satisfies Eq. (\ref{mainiteration}) for some $w(\cdot, \cdot, \cdot, \cdot, \cdot)$ then all or almost all $x_t(\alpha)$ converge as $t \rightarrow \infty$. It is not even true that the distribution of values of $x_t(\alpha)$ converges in any meaningful sense: we see from the above example that a mass of agents can alternate between being close to $z_1(0)$ and close to $z_2(0)$. }

\aor{However}, we will show that under the assumption that $w(\cdot, \cdot, \cdot, \cdot, \cdot)$ is {\em symmetric}, i.e., 
\[ w(t, \alpha, \beta, x_t(\alpha), x_t(\beta))  = w(t, \beta, \alpha, x_t(\beta), x_t(\alpha)) \] the {\em distribution} of the agents is guaranteed to converge. We
state this result formally next. 

First, we review the relevant notion of convergence of distributions. Every function $x_t(\alpha)$ defines a measure $\mu_t$ on the real line in the natural way \begin{equation} \label{mudef} \mu_t(A) = {\mathcal L}( \{ \alpha ~|~ x_t(\alpha) \in A \}) \end{equation} where ${\mathcal L}(\cdot)$ refers to the Lebesgue measure
on the real line. These measures $\mu_t$ are a natural way to summarize the concentrations of the values $x_t(\alpha)$. The most natural definition of convergence of distributions would be that $\lim_t \mu_t(A) = \mu_{\infty}(A)$ for any Borel measurable set $A$; this is referred to as {\em strong convergence} of measures and it is usually too restrictive to be used in practice (under this definition, for example, defining $\mu_{y}(A) = 1_{y \in A}$ we have that it is not true that $\mu_{1/t}$ converges to $\mu_0$ as $t \rightarrow \infty$). Consequently, the most commonly used notion is {\em convergence in distribution} (sometime referred to in this context as weak-$*$ convergence): $\mu_t$ approaches $\mu_{\infty}$ in distribution if 
$\lim_t \mu_t(A) = \mu_{\infty}(A)$ for all sets $A$ whose boundary has measure $0$ under $\mu_{\infty}$. \jmh{Equivalently, $\mu_t$ approaches $\mu_\infty$ in distribution if $\lim_{t\to \infty} \int \eta(x) d\mu_t(x) = \int \eta(x) d\mu_\infty$ for every bounded continuous function $\eta$. This equivalence is part of Portmanteau's theorem, see for example \cite{b86}.
}

We can now state our first main result. 

\smallskip

\begin{theorem} \label{weakconv} Suppose $w(\cdot, \cdot, \cdot, \cdot, \cdot)$  is nonnegative and symmetric and let $x_t(\alpha)$ be a solution of Eq. (\ref{mainiteration}) and  $\mu_t$ be defined as in Eq. (\ref{mudef}). Then there exists a measure $\mu_{\infty}$ on $[0,1]$ such that $\mu_t$ approaches $\mu_{\infty}$ in distribution. \end{theorem}  

\smallskip

Theorem \ref{weakconv} states that symmetry is sufficient for convergence (though recall that we made several assumptions in order for Eq. (\ref{mainiteration}) to make sense which we are not mentioning). \jmh{We will see that it can directly be extended to multi-dimensional opinions}.
It is \jmh{significantly} stronger than the results previously available in the 
literature. For example, \cite{cft12} proves a version of Theorem \ref{weakconv} in a similar \aor{(discrete-time)} model under the additional assumptions that $w(\cdot, \cdot, \cdot, \cdot, \cdot)$ depends only on $x_t(\beta)-x_t(\alpha)$ (so it is a function of just one argument). Similarly, \cite{bht09} proves a version of Theorem \ref{weakconv} under the assumption that the solution $x_t(\alpha)$ is monotonic.

We next turn to the question of whether it is possible to improve upon the conclusion of convergence in distribution. In fact, convergence in distribution does not seem to be the most natural notion of convergence for this class of systems; the more plausible notion would be {\em convergence of almost all agents}, i.e., that $x_t(\alpha)$ converges for almost all $\alpha$. We show, however, that convergence in this sense may not occur. 

\smallskip

\begin{theorem} \label{counterexample} There exists a symmetric, nonnegative $w(\cdot, \cdot, \cdot, \cdot, \cdot)$ and \jmhr{bounded} functions $x_t$ satisfying Eq. (\ref{mainiteration}) with this 
$w(\cdot, \cdot, \cdot, \cdot, \cdot)$ such that the set of $\alpha$ where $x_t(\alpha)$ does not converge has positive measure. 
\end{theorem}

\smallskip

We believe that the contrast between Theorem \ref{weakconv} and Theorem \ref{counterexample} is somewhat surprising. Indeed, for opinion dynamics of finitely many agents, it is not hard to see that convergence in distribution and of all agents are equivalent in continuous time but not in discrete time\footnote{More formally, if $L$ is a Laplacian matrix then $x'(t)=-Lx(t)$ converges in distribution if and only if each $x_i(t)$ converges; on the other hand, there exist symmetric stochastic matrices $A$ such that $x(t+1)=A x(t)$ has e the property that some $x_i(t)$ diverges while the distribution of the vector $x(t)$ remains fixed.}. Our initial conjecture when beginning this work was that for a continuum of agents with symmetric kernel, convergence will happen either in both senses or not at all; however, these two theorems show this to be false.

Nevertheless, in spite of Theorem \ref{counterexample}, it may be possible that a stronger form of convergence holds in many natural settings. The extent to which this is true remains an open question. But, we are able to show that as long as agents with similar opinions exert
a non-negligible pull on each other, almost all of the $x_t(\alpha)$ will indeed converge. Formally, let us define 
the set of functions $\Gamma(r,\Delta)$ for $r>0, \Delta>0$ to be the set of $w(t, \alpha, \beta, x_t(\alpha), x_t(\beta))$ which satisfy the following property: if \[ |x_t(\alpha)-x_t(\beta)| \leq r \]  then \[ w(t, \alpha, \beta, x_t(\alpha), x_t(\beta)) \geq \Delta. \] We then have the following theorem. 

\smallskip

\begin{theorem}\label{densityconvergence} If $w(\cdot, \cdot, \cdot, \cdot, \cdot)$ is symmetric, nonnegative, and belongs to some $\Gamma(r, \Delta)$ for $r>0, \Delta>0$, then $x_t(\alpha)$ converges for almost all $\alpha$. 

\jmh{Moreover, there exist a finite number of points $z_1,z_2,\dots,z_k\in [0,1]$ with $\abs{z_i-z_j}\geq r$ for all $i\neq j$ such that $\lim_{t\to\infty} x_t(\alpha) \in \{z_1,\dots,z_k\}$ for almost all $\alpha$.}
\end{theorem}

\smallskip

We can also guarantee convergence for almost all agents when the interactions only depend on time and on the opinions via a sufficiently smooth and bounded function.

\providecommand{\tw}{\tilde w}

\begin{theorem}\label{thm:Lips->conv}
Suppose that the interaction weights $w(.,.,.,.,.)$ are symmetric, nonnegative, and depend only on time and the positions:
$$w(t,\ao{\alpha, \beta}, x_t(\alpha),x_t(\beta))  = \tw(t,x_t(\alpha),x_t(\beta)).$$
If $\tw$ is 
Lipschitz continuous, i.e. $|\tw(t,x,y) - \tw(t,x,z)| \leq L \abs{y-z}$ for all $t,x,y,z$, and some $L$,
then any soluton $x_t(\alpha)$ of (\ref{mainiteration}) converges for almost all $\alpha$.
\end{theorem}

We note that as a particular case of both results above, convergence of $x_t(\alpha)$ for almost all $\alpha$ is guaranteed for the one-dimensional versions of the systems considered in \cite{cft12}, where the authors proved convergence in distribution.

\subsection*{Outline}

The rest of this paper is dedicated to providing proofs of these \ao{four} theorems. Theorem \ref{weakconv} on convergence in distribution is proved in Section \ref{weaksection}. Theorem \ref{counterexample} on the possible absence of convergence of the opinions themselves is proved in Section \ref{counterproof}, while the positive convergence results, Theorems \ref{densityconvergence} and \ref{thm:Lips->conv} are proved in Section \ref{strongsection}. We briefly summarize and describe some open questions in
Section \ref{conclusion}.\\
\aor{Finally, although a detailed analysis of questions of existence and uniqueness is outside the scope of the present paper, we show in Appendix A that existence and uniqueness does hold in a substantial class of models,} \jmhr{and therefore that our convergence results do apply to actual solutions in those cases}.

\section{Convergence in distribution\label{weaksection}}  In this section we will prove Theorem \ref{weakconv}. Our proof is quite short and rests on a novel combination of two techniques: a simple exchange of integrals to establish that {\em every} convex function is a Lyapunov function (see proof of Lemma \ref{convdec}) coupled with an appeal to some results of Haussdorff about the moment problem (see proof of Theorem \ref{weakconv} below). But first, we prove a technical lemma stating that, consistently with the intuition, all opinions remain in $[0,1]$ and their evolution is uniformly Lipschitz continuous.

\smallskip

\begin{lemma}\label{lem:x_[01]_Lips}
Let $x$ be a solution of \eqref{mainiteration}. Then $x_t(\alpha) \in [0,1]$ for all $\alpha\in[0,1],t\geq 0$, and $\abs{x_t(\alpha)-x_s(\alpha)}\leq W\abs{s-t}$ for all $s,t\geq 0$.
\end{lemma}
\begin{proof}
Suppose, to obtain a contradiction, that there exists a time $t^*$ and some agent $\alpha^*$ for which $x_{t^*}(\alpha^*)<0$. Let then $- m := \inf_{\beta\in [0,1], t\in [0,t^*]} x_s(\beta)$ be the infimum of all opinions that have been held until time $t^*$. Our assumption that $x$ does not explode in finite time implies that $m$ is finite. 

Consider now an agent $\alpha$. Since $x_t(\alpha)$ is absolutely continuous with respect to $t$ and satisfies \eqref{mainiteration} for almost all $t$, there holds
$$
x_t(\alpha) - x_0(\alpha) = \int_{0}^{t} \left(\int_{\beta \in [0,1]} w(s,\alpha,\beta,x_s(\alpha),x_s(\beta)) (x_{s}(\beta)-x_s(\alpha))d\beta \right)ds
$$
for every $t$. It follows then from the definition of $m$ and the (uniform) boundedness and nonnegativity of $w$ that for every $t\in [0,t^*]$,
\begin{align*}
x_t(\alpha) - x_0(\alpha) &\geq \int_{0}^{t} \left(\int_{\beta \in [0,1]} w(s,\alpha,\beta,x_s(\alpha),x_s(\beta)) (-m-x_s(\alpha))d\beta \right)ds
\\&
\geq \int_{0}^{t} \left(\int_{\beta \in [0,1]} W .(-m-x_s(\alpha))d\beta \right)ds= \int_{0}^{t} W. (-m-x_s(\alpha)) ds,
\end{align*}
where we have used $-m-x_s(\alpha)\leq 0$ in the second inequality.
An integral form of Gronwall's inequality (see Section 1.1 in \cite{D03}) implies then that 
$$
\jmhr{x_t(\alpha)} \geq - m + (x_0(\alpha) + m) e^{-W t}, \hspace{.2cm} \forall t\leq t^*.
$$
Remembering that $x_0(\alpha)\geq 0$, we obtain then that  $x_t(\alpha) \geq - m + me^{-Wt^*}> -m$ holds 
for every $\alpha$ and $t\in [0,t^*]$, in contradiction with $0> -m = \inf_{t \in [0,t^*] , \alpha\in [0,1]}x_t(\alpha)$. Our hypothesis that $x_t(\alpha) <0 $ for some $t^*$ and $\alpha$ does thus never hold. A symmetric argument shows that $x_t(\alpha)\leq 1$ for every $\alpha$ and $t\geq 0$. As a consequence we also have $\abs{x_t(\alpha)-x_t(\beta)}\leq 1$, and the Lipschitz continuity condition $\abs{x_t(\alpha)-x_s(\alpha)}\leq W\abs{s-t}$ follows then directly from \eqref{mainiteration} and the bound $w(.,.,.,.,.)\leq W$. 
\end{proof}

\smallskip

\begin{definition}  Given a Borel measurable function $f: \R \rightarrow \R$ define \[ V_f(t) = \int_{0}^{1} f(x_t(\alpha)) ~d \alpha \] 
\end{definition}

\begin{lemma} \label{convdec} If $\ao{f: \R \rightarrow \R}$ is convex \jmh{and differentiable} \ao{everywhere} then for all $t \geq 0$, $\dot{V_f}(t) \leq 0$. \end{lemma}

\smallskip

\begin{proof} 
\jmh{In what follows, we use the abbreviation $\tilde w_{z,\alpha}$ to denote $w(t, z, \alpha, x_t(z), x_t(\alpha))$.} We argue as follows. 
\ao{First,  $V(t)$ is an absolutely continuous function which satisfies \[\dot{V}_f(t) = \frac{d}{dt} \int_0^1 f(x_t(\alpha)) ~ d \alpha = \int_0^1 \frac{d}{dt} f(x_t(\alpha)) ~ d \alpha \]  for almost all $t$. We justify this as follows. Observe that (i) $x_t(\alpha)$ is jointly measurable in $t$ and $\alpha$ because it is  measurable in $\alpha$ for fixed $t$ and Lipschitz continuous in $t$ for fixed $\alpha$ (by Lemma \ref{lem:x_[01]_Lips}). Since  $f(\cdot)$ is differentiable, $f(x_t(\alpha))$ is jointly measurable in $t$ and $\alpha$ as well. (ii) Because $x_t(\alpha)$ is Lipschitz continuous, we have that $f(x_t(\alpha))$ is integrable over $\alpha$ for any fixed $t$; moreover, for the same reason, $f(x_t(\alpha))$ is absolutely continuous with respect to $t$ for each $\alpha$. (iii) The almost everywhere time derivative of $f(x_t(\alpha))$ which is
$f'(x_t(\alpha)) \dot{x}_t(\alpha)$ is integrable over any compact subset of $(0,\infty) \times [0,1]$. It can then be proved (see for example Theorem 3 of \cite{pm}) that (i)-(iii) implies that the function $V_f(t)$ is absolutely continuous and satisfies the above equation for almost all $t$.}

\ao{Consequently,} 
\begin{eqnarray*}\dot{V}_f(t)
& = & \int_{[0,1]^2} f'(x_t(\alpha))  \tilde w_{z,\alpha} ( x_t(z) -   x_t(\alpha) )  \aor{~dz ~d \alpha} \\
& = & \frac{1}{2} \int_{[0,1]^2} \tilde w_{z,\alpha}f'(x_t(a)) ( x_t(z) -   x_t(\alpha)  ) \aor{~dz ~d \alpha}  \\ && + \frac{1}{2} \int_{[0,1]^2}\tilde w_{z,\alpha}
f'(x_t(z)) ( x_t(\alpha) -   x_t(z) )  \aor{~dz ~d \alpha}  \\
& = & \frac{1}{2} \int_{[0,1]^2} \tilde w_{z,\alpha}(f'(x_t(\alpha)) - f'(x_t(z))) (x_t(z) - x_t(\alpha))  \aor{~dz ~d \alpha}
\end{eqnarray*}  \ao{Here we applied Fubini's theorem relying on the boundedness of $x_t(\alpha)$ for any fixed $t$, the boundedness of  $w(\cdot, \cdot, \cdot, \cdot, \cdot)$, and finally the boundedness $f'$ on the closed interval $[0,1]$.}

\ao{Now examining the final expression, we see that} since $w(\cdot, \cdot, \cdot, \cdot, \cdot)$ is nonnegative and $f'(\cdot)$ is an increasing function due to the convexity of $f$, the integrand is always nonpositive. \ao{The absolute continuity of $V_f(t)$ then implies that $V_f(t)$ is nonincreasing.}
\end{proof}

\smallskip

\begin{remark} \label{momentsremark} As a consequence of Lemma \ref{convdec}, the functions $m_t(k) = \int_0^1 x_t(\alpha)^k ~d \alpha$ are nonincreasing for each $k \geq 1$. Viewing \ao{each} $x_t$ as a random variable with state-space $[0,1]$, $m_t(k)$ has the interpretation that it is the $k$'th moment of this random variable.  
\end{remark}

\jmh{\vspace{.2cm}
We can now prove the main result of this section.
\vspace{.2cm}
}

\begin{proof}[Proof of Theorem \ref{weakconv}] Let us view each $x_t$ as a random variable defined on the state-space $\Omega=[0,1]$. Remark \ref{momentsremark} implies that the moments $m_t(k)$  are nonincreasing for $k \geq 1$; consequently, $\lim_{t \rightarrow \infty} m_t(k)$ exists for each $k \geq 1$. We next argue that there exists a random variable $x_{\infty}$ whose $k$'th moment is  \jmh{$\lim_{t\to\infty} m_t(k)$}. This follows from a result of Haussdorff, which is that a sequence $s = s_0, s_1, s_2, \ldots$ is a valid moment sequence for a random variable with values in $[0,1]$ if and only if a certain infinite family of linear combinations of the $s_i$ are nonnegative; each of these combinations has only finitely many nonzero coefficients (see \cite{st43}, Theorem 1.5). 
\jmh{Clearly, the moments of each $x_t$ satisfy this condition since each $x_t$ defines such a random variable. Moreover, every moment remains in a compact set independent of $t$. Therefore the limiting moments $\lim_{t\to\infty} m_t(k)$ also satisfy the condition of Haussdorff's result.}
We conclude that there is a random variable $x_{\infty}$ \ao{taking values in $[0,1]$} such that the moments of $x_t$ converge to the moments of $x_{\infty}$. \ao{Naturally, all the moments of $x_{\infty}$ are in $[0,1]$.}

Next, convergence of moments $x_t$ to the moments $x_{\infty}$ implies convergence in distribution of $x_t$ to $x_{\infty}$ if \ao{the distribution of} $x_{\infty}$ is uniquely defined by its moments (\cite{b86}, Theorem 30.2). However, any random variable whose \ao{moments are in $[0,1]$}  \ao{has the 
property that its distribution is}  uniquely defined by its moments (\cite{b86}, Theorem 30.1). Thus $x_t$ converges to $x_{\infty}$ in distribution.  By the Portmanteau theorem (\cite{p02}, Section 7.1), this immediately implies the conclusion of this Theorem.
\end{proof}

\smallskip

\jmh{As a final remark about Theorem \ref{weakconv}, we note that even though it is stated for one-dimensional opinions ($x_t(\alpha)\in \R$), it can directly be extended to opinions in $x_t(\alpha)\in\R^q$ satisfying a $q-$dimensional version of \eqref{mainiteration} provided that the interactions weights $w$ remain scalar or $q-$dimensional diagonal matrices. Indeed, one can in that case apply Theorem \ref{weakconv} to each component of $x_t(\alpha)$ separately.}

\section{Absence of convergence}\label{counterproof} \ao{This section is dedicated to proving Theorem \ref{counterexample} showing that, intriguingly, 
there exist opinion dynamics systems converging in distribution for which a positive measure set of agents does not converge.}

\subsection{Proof sketch} 
\ao{The proof, while somewhat technically involved, is based on a simple idea. We consider a situation in which there is always a constant mass  of agents {\sl uniformly distributed} between $[-1/2, 1/2]$ with every agent constantly cycling between $-1/2$ and $1/2$. We will show that this cycling can take place at a speed $v(t)$ which drops to zero with time but slowly enough so that every agent makes the loop from $-1/2$ to $1/2$ infinitely many times.} 

\ao{We accomplish this in terms of Eq. (\ref{mainiteration}) as follows. There will be a mass of agents uniformly distributed in $[-1/2,1/2]$ moving right; and similarly mass of agents uniformly distributed in $[-1/2,1/2]$ moving left. When each agent reaches $\pm 1/2$, it switches from
one mass to the other. Agents moving right put a positive weight on agents {\sl to their right moving left} while 
agents moving left put a positive weight on agents {\sl on their left moving right}.}

\ao{ Specifically, every agent moving right will put a constant weight on all agents moving left and located within an interval of size $O(\sqrt{v(t)})$ to its right; and similarly every agent moving left will put a constant weight on all agents moving right and located in an interval of size $O(\sqrt{v(t)})$ to its left.  The idea is that an agent moving right will be pulled by $O(\sqrt{v(t)})$ agents at an average distance of $O(\sqrt{v(t)})$, so that the average pull by each of the agents will be of strength $O(\sqrt{v(t)})$ because, $w(\cdot, \cdot, \cdot, \cdot, \cdot)$ being constant, the pulls agents exert on each other is proportional to their distance. Thus its velocity will come out to be $v(t)$. However, this will not work for those agents which are closer than $O(\sqrt{v(t)})$ to $1/2$ and $-1/2$ which do not have enough agents to their right and left, respectively.}

\ao{We fix this by introducing a unit mass of agents $C_R$ initially located at $C_0>1/2$ and a unit mass of agents $C_L$ initially located at $-C_0<-1/2$. Agents in $[-1/2,1/2]$ within a distance $O(\sqrt{v(t)})$ of $1/2$ will put a positive weight on agents in $C_R$ while agents in $[-1/2,1/2]$ within the same
distance of  $-1/2$ will put
a positive weight on agents in $C_L$. The weights will be precisely chosen so that the agents within $O(\sqrt{v(t)})$ of $1/2$ and $-1/2$ will cycle with speed $v(t)$, as the other agents in $[-1/2,1/2]$.  Note that
because the weights are symmetric, this has the effect of pushing the agents  in $C_R$ to the left over time and similarly pushing the 
agents in $C_L$ right over time. The user may refer to Figure \ref{fig:counterex} for a graphical representation.}

\ao{{\sl The key point is that the weights $w(t,\alpha, \beta, x_t(\alpha), x_t(\beta))$ between, say, $ \alpha \in C_R$ and $\beta$ in the right corner of $[-1/2,1/2]$ only need to be of magnitude $O(v(t))$ to make $\beta$ cycle with speed $v(t)$.} This is simply because there is a unit mass of agents in $C_R$ and so as long as the separation between $[-1/2,1/2]$ and the agents in $C_R$ is bounded away from zero, which we will show can be achieved, the weight needed to pull the agent $\beta$ with speed $v(t)$ to the right is $O(v(t))$. Similarly, the weights between $C_L$ and the left corner of $[-1/2,1/2]$ need to be of magnitude $O(v(t))$. }

\ao{Consequently, agents in $C_R$ move left with speed $O(\sqrt{v(t)}) \cdot O(v(t)) = O ( v(t)^{1.5})$, since they get pulled left by a mass of 
agents of size $O(\sqrt{v(t)})$ on which they put weights of magnitude $O(v(t))$. If $v(t)$ is chosen so that $\int_0^{\infty} v(t)  = + \infty$ while $\int_0^{\infty} v(t)^{1.5}$ is finite and sufficiently small,  the agents in $C_R$ do not approach $+1/2$ but rather remain bounded away from it. Similarly, the agents in $C_l$ remain bounded away from $-1/2$. Thus the motion we
have just described - agents in $C_R$ moving left, agents in $C_L$ moving right, and agents between $[-1/2,1/2]$ cycling - goes on forever, and the 
positive mass of agents cycling in $[-1/2,1/2]$ never converge, even though the speed of the agents decay to 0 and the system converges in distribution consistently with Theorem \ref{weakconv}. }

\subsection{The proof} \ao{Without further ado, we proceed to the technical details of the argument. }

\smallskip
\begin{proof}
We first describe an evolution $x$ for which a positive measure set of agents does not converge. 
We then define some symmetric weights $w(t,\alpha,\beta)$, and show that this evolution $x$ is a solution of (\ref{mainiteration}) for those weights.

For the clarity of exposition, the agents are separated in three distinct sets, and $x_0$ is bounded but its image is not included in $[0,1]$. The first two sets are $C_R$ and $C_L$, for right and left cluster respectively, both of measure 1. The third set is $I=[0,2]$. Our system can be recasted in the form (\ref{mainiteration}) with agents indexed on $[0,1]$ by a simple change of variable. Besides, when there is no risk of ambiguity, we will sometimes refer to the agents as moving instead of as having their opinions changing, and call \quotes{velocity} the variation rate $\dot x_t(\alpha)$ of their opinions.

\vspace{.3cm}
\noindent\emph{A. Definition of $x$}
\vspace{.2cm}

\begin{figure}
\centering
\includegraphics[scale=.42]{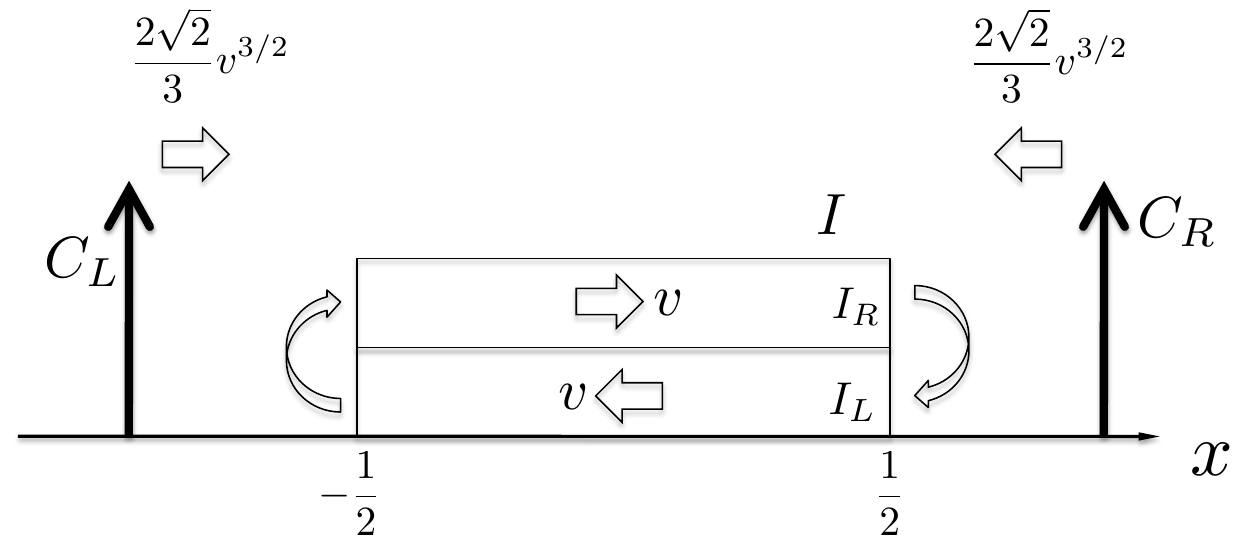}

\caption{Representation of the density of agents corresponding to the non-converging evolution $x$ used in the proof of Theorem \ref{counterexample}. The agents in $I=I_R\cup I_L$ continuously move from left to right, with a speed $v$ whose integral does not converge. The clusters $C_L$ and $C_R$ on the other hand move with a speed proportional to $v^{3/2}$, the integral of which does converge, so that they never reach $[-\frac{1}{2},\frac{1}{2}]$.
}\label{fig:counterex}
\end{figure} 
The evolution $x$ is represented in Fig. \ref{fig:counterex}. 

Formally, we select a function $v:\R^+\to\R^+$ such that $\int_0^\infty v(t)dt$ diverges, but $\int_0^\infty v^{3/2}(t)dt <\infty$ (for example $v(t) = (t+1)^{-3/4}$). 
Let then $S(t) = -\frac{1}{2}+\abs{1-\prt{t \mod  2 }}$, where $t \mod 2$ denotes the unique value in $[0,2)$ equal to $t-2k$ for some integer $k$. 
Observe that $S(t)\in [-\frac{1}{2},\frac{1}{2}]$, and that $\frac{dS}{dt}(t)$ is defined everywhere except for integer $t$, with $\frac{dS}{dt}(t) =1$ when $\floor{t}$ is odd and $-1$ when $\floor{t}$ is even. We now define $x$, represented in Fig. \ref{fig:counterex}, in the following way:

\begin{itemize}
\item{If $\alpha\in I$, then $x_t(\alpha) = S\prt{\alpha + \int_0^tv(s)ds}$.}
\item{If $\alpha \in C_L$, then\\ $x_t(\alpha)= x_t(C_L):= -C_0 + \int_0^t \frac{2\sqrt{2}}{3}v(s)^{3/2}  ds$.}
\item{If $\alpha \in C_R$, then\\ $x_t(\alpha) = x_t(C_R) := C_0 - \int_0^t \frac{2\sqrt{2}}{3}v(s)^{3/2}  ds$},
\end{itemize}
for some $C_0> 1/2 +  \int_0^\infty \frac{2\sqrt{2}}{3}v(s)^{3/2}  ds. $ 

Observe that $x$ does not converge, as the agents in $I$ keep moving with speed $v$, the integral of which diverges. The evolution $x$ converges however in distribution, consistently with Theorem \ref{weakconv}. Indeed, one can verify that the density of agents on $[-1/2,1/2]$ remains constant at $2$ for all $t$. 
At the same time, the opinions of the agents in $C_L$ and $C_R$ converge monotonously  to some points out of the interval $[-1/2,1/2]$ due to our assumption on $C_0$.

In order to describe precisely the velocities and to later define the interaction weights, it is convenient to separate $I$ in two time-varying disjoint sets. For every $t$, we let $I_R(t)$ be the set of those $\alpha\in [0,2]$ for which $\floor{\alpha + \int_0^t v(s)}$ is odd, and $I_L(t)$ the set of those for which this expression is even. Observe that since the shift $\int_0^t v(s)$ is the same for all agents, which are all initially in $[0,2]$, $I_R(t)$ and $I_L(t)$ both have the same measure 1, and the \quotes{density} of agents of $I_R$ and $I_L$ on $[-\frac{1}{2},\frac{1}{2}]$ are both uniform and equal to 1 at all time. Besides, one can verify using the definition of $x_t(\alpha)$ on $I$ that $\dot x_t(\alpha) = v(t)$ 
for every $\alpha \in I_R(t)$ and $\dot x_t(\alpha) = -v(t)$ for every $\alpha \in I_L(t)$ (except for the two agents for which $\alpha + \int_0^t v(s)$ is an integer, which we will neglect in this proof). Similarly, $\dot x_t(\alpha) = \frac{2\sqrt{2}}{3}v(t)^{3/2}$ for every $\alpha \in C_L$, and $\dot x_t(\alpha) = -\frac{2\sqrt{2}}{3}v(t)^{3/2}$ if $\alpha \in C_R$. These four equalities actually entirely characterize $\dot x_t(\alpha)$ (up to the two points where it might not be defined).

\vspace{.3cm}
\noindent\emph{B. Interaction weights}
\vspace{.2cm}

We now define some symmetric weights $w$ in order to later show that $x_t$ is a solution to (\ref{mainiteration}) for these weights. The idea is that agents in $I$ will move thanks to an attraction exerted by agents moving in the other direction and lying in an interval of length $\epsilon(t)$ ahead of them, for an appropriately chosen $\epsilon(t)$. Agents at distance less than $\epsilon$ from the edge of the interval to which they are moving are in addition attracted by the clusters $C_L$ or $C_R$, as represented in Fig. \ref{fig:counterex_interactions}. These clusters are attracted in return, but move sufficiently slowly so that they never reach the interval $[-\frac{1}{2},\frac{1}{2}]$.

\begin{figure}
\centering 
\includegraphics[scale = .5]{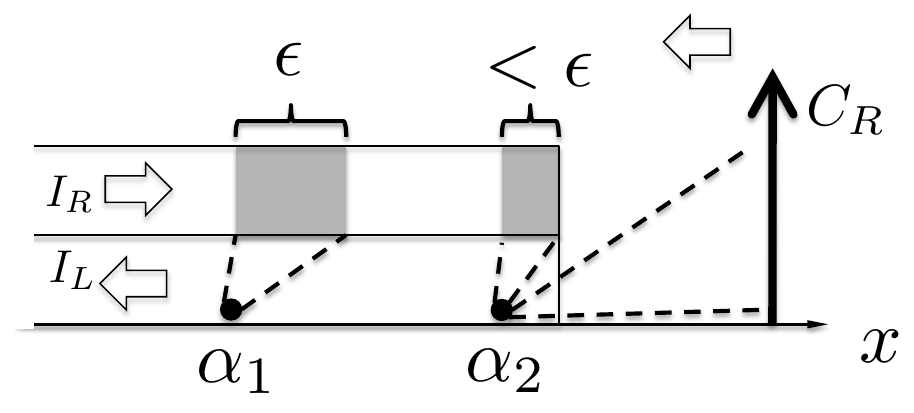}
\caption{Representations of the interactions in the example used in the proof of Theorem \ref{counterexample}. The agent $\alpha_1$ moves to the right hand side, and is attracted by all the agents moving on the left hand side that are in front of it and distant by less than $\epsilon$ form it, guaranteeing a speed $v$. The agent $\alpha_2$ is a distance smaller then $\epsilon$ from the edge of the distribution, so there are not enough agents to attract it with a speed $v$ (if interactions have the same intensity). This is compensated by interactions with the cluster $C_R$.   }\label{fig:counterex_interactions}
\end{figure}

The interactions between the agents of $I$ are defined by 
\begin{equation*}
w(t,\alpha,\beta) = 1  \text{ if }    \alpha \in I_R(t),\beta \in I_L(t) \text{ and } 0  < x_t(\beta)-x_t(\alpha) < \epsilon(t),
\end{equation*}
and 0 else, 
where we remind the reader that all weights are symmetric. In addition, agents of $I_R(t)$ with opinions above $\frac{1}{2}-\epsilon(t)$ interact with those of $I_R(t)$ with weights
\begin{equation}\label{eq:weights_right_cluster}
w(t,\alpha,\beta) = \frac{\epsilon(t)^2 -\prt{\frac{1}{2}- x_t(\alpha) }^2}{2\prt{x_t(C_L) - x_t(\alpha) }} 
\end{equation}
for $\alpha \in I_R(t),$ $\beta\in C_R$ and $x_t(\alpha) \geq \frac{1}{2}-\epsilon(t).$
Similarly,
\begin{equation*}
w(t,\alpha,\beta) = \frac{\epsilon(t)^2 -\prt{\frac{1}{2} + x_t(\alpha) }^2}{2\prt{x_t(\alpha) - x_t(C_R) }} 
\end{equation*}
if $\alpha \in I_L(t),$ $\beta\in C_L$ and $x_t(\alpha) \leq - \frac{1}{2}+\epsilon(t).$

\vspace{.3cm}
\noindent\emph{C. $x$ solution of (\ref{mainiteration}) for these weights $w$.}
\vspace{.2cm}

To show that $x$ is a solution of the system (\ref{mainiteration}) for these weights, 
we just need to show that the right-hand side term of (\ref{mainiteration}), $\int  w(t,\alpha,\beta)(x_t(\beta)-x_t(\alpha))d\beta$, is equal to the velocities $\dot x_t(\alpha)$ computed in part A.  This expression should thus be equal to $v(t)$ (resp. $-v(t)$) for $\alpha\in I_R(t)$ (resp. $I_L(t)$), and to $\frac{2\sqrt{2}}{3}v(t)^{3/2}$ (resp. $-\frac{2\sqrt{2}}{3}v(t)^{3/2}$) for $\alpha \in C_L$ (resp. $C_R$), neglecting again the two agents for which $\dot x_t(\alpha)$ is not defined.

We first consider an agent $\alpha\in I_R(t)$ distant from the boundary $\frac{1}{2}$ by more than $\epsilon(t)$, and interacting thus only with those agents of $I_L(t)$ having opinions between $x_\alpha(t)$ and $x_\alpha(t) + \epsilon(t)$. Its velocity $\dot x_t(\alpha)$ is
$$
\int_{\beta\in I_L(t):x_t(\beta) \in [x_t(\alpha),x_t(\alpha) + \epsilon(t)]}w(\alpha,\beta,t) \prt{x_t(\beta)-x_t(\alpha)} d\beta.
$$
Since the density of $I_L(t)$ is $1$ over the whole interval $(-\frac{1}{2},\frac{1}{2})$, we have
\begin{equation}\label{eq:alpha_normal}
\dot x_t(\alpha) = \int_{y = x_t(\alpha)} ^{x_t(\alpha)+\epsilon(t)}\prt{y-x_t(\alpha)}  dy = \frac{1}{2}\epsilon(t)^2  = v(t),
\end{equation}
consistently with the definition of $x$. Let us now consider an agent  $\alpha\in I_R(t)$ distant from the boundary $\frac{1}{2}$ by less than $\epsilon(t)$. Such agents interact with agent in $I_L(t)$ and 
with the agents in $C_R$. There holds thus 
\begin{align*}
\dot x_t(\alpha) =\int_{\beta\in I_L(t):x_t(\beta) \in [x_t(\alpha),x_t(\alpha) +  \epsilon(t)]}w(\alpha,\beta,t) \prt{x_t(\beta)-x_t(\alpha)} d\beta \\
+ \int_{\gamma \in C_R} w(\alpha,\beta,t) \prt{x_t(\gamma) - x_t(\alpha)}d\gamma.
\end{align*}
Taking into account that the density of $I_R(t)$ is $1$ in $(-\frac{1}{2},\frac{1}{2})$, the fact that $x_t(\alpha)>\frac{1}{2}-\epsilon(t)$ and the fact that all agents in $C_R$ have the same opinion $X_t(C_r)$, and interact with $\alpha$ according to (\ref{eq:weights_right_cluster}), we obtain again \begin{small}
\begin{align*}
\dot x_t(\alpha)= & \int_{y = x_t(\alpha)} ^{1/2}\prt{y-x_t(\alpha)}  dy + \frac{\epsilon(t)^2-\prt{\frac{1}{2}- x_t(\alpha) }^2}{2\prt{x_{t}(C_R) - x_t(\alpha) } }(x_t(C_R) - x_t(\alpha)) \int_{\gamma \in C_r}d\gamma \\ =  &
\frac{1}{2}\prt{\frac{1}{2}-x_t(\alpha)}^2 + \frac{1}{2}\prt{\epsilon(t)^2-\prt{\frac{1}{2}- x_t(\alpha) }^2}   = \frac{1}{2}\epsilon(t)^2 =v(t).
\end{align*} \end{small}
We now consider an agent $\alpha$ in the right cluster $C_R$. All agents in $C_R$ have the same opinion $x_t(C_R)$ and are interacting with agents in $I_R(t)$ at a distance less than $\epsilon(t)$ from $1/2$, with the weights (\ref{eq:weights_right_cluster}), so that $\dot x_t(\alpha) =$
$$
\int_{\beta \in I_R(t): x_t(\beta) \geq \frac{1}{2} -\epsilon(t)} 
\frac{\epsilon^2- \prt{\frac{1}{2}- x_t(\beta) }^2}{2\prt{x_{t}(C_R) - x_t(\alpha) } }(x_t(\beta) - x_t(C_R)  )
d\beta 
$$
Using again the the density of $I_R(t)$ is 1 over $(-\frac{1}{2},\frac{1}{2})$, we obtain
\begin{align*}
\dot x_t(\alpha) &= -\frac{1}{2}\int_{y= 1/2- \epsilon}^{1/2} \prt{\epsilon^2- \prt{\frac{1}{2}- y}^2}dy = \frac{1}{3}\epsilon(t)^3 =- \frac{2\sqrt{2}}{3}v(t)^{3/2}.
\end{align*}
consistently with our definition of $x$. We have thus proved that $x$ satisfies (\ref{mainiteration}) with the symmetric weights that we have defined for all agents in $I_R(t)$ and $C_R$. A symmetric argument applies to agents in $I_L(t)$ and $C_L$, so that $x$ is indeed a solution of  (\ref{mainiteration}), which achieves the proof, since we have already seen that $x_t(\alpha)$ does not converge for any $\alpha \in I = I_L(t) \cup I_R(t)$.
\end{proof}

\section{Convergence for almost all agents\label{strongsection}} 

In this section we will prove Theorems \ref{densityconvergence} and \ref{thm:Lips->conv}, providing in both theorems sufficient conditions for the convergence of almost all agents. As we will see, proving these theorems requires us to go beyond establishing the existence of limits for moments of $x_t(\alpha)$. Rather, we will need to derive explicit convergence rates for these moments and argue that only certain patterns of motion on the part of the agents are consistent with those
rates.

\subsection{Proof of Theorem \ref{densityconvergence}\label{densproof}} 

\jmh{Our proof has two parts. First we prove that the mass of agent opinions eventually concentrates on a finite number of points $z_1,\dots,z_k$. Then we will argue that no agent, except possibly those in a zero measure set, can move from the neighborhood of one of these points to the neighborhood of another infinitely often.}

We begin with the key lemma which explores a consequence of the decay of the second moment $m_2(t)$. \ao{We remind the reader that all the
assumptions made in the introduction are assumed to hold in this section.}

\smallskip

\begin{lemma} \label{finiteintegrability} \[ \int_0^{\infty} \int_{\jmh{[0,1]^2}} w(t, \alpha, \beta, x_t(\alpha), x_t(\beta)) (x_t(\beta) - x_t(\alpha))^2 ~d \alpha ~d \beta ~dt < \infty \]
\end{lemma}

\smallskip

\begin{proof} 
Recall that $m_2(t) =\int_0^1 x_t(\alpha)^2 ~d \alpha$. \jmh{Using again the abbreviation $\tilde w_{z,\alpha}$ to denote $w(t, z, \alpha, x_t(z), x_t(\alpha))$,} we have,  
\begin{eqnarray*} \dot{m}_2(t) & = & 2 \int_{[0,1]} x_t(\alpha) \dot{x}_t(\alpha) ~d \alpha \\
& = & 2 \int_{[0,1]^2} \tilde w_{z,\alpha} x_t(\alpha) x_t(z)\aor{~dz ~d \alpha} - 2 \int_{\ao{[0,1]^2}}\tilde w_{z,\alpha}  x_t^2(\alpha) \aor{~ dz ~ d \alpha} \\ 
& = & - \int_{[0,1]^2} \tilde w_{z,\alpha} ( x_t(\alpha) - x_t(z))^2   \aor{~ dz ~ d \alpha} 
\end{eqnarray*}  where the final step used the symmetry of $w$. We remark that this derivation closely parallels the proof of Lemma \ref{convdec} \ao{and  the same justifications for differentiating under the integral sign and using Fubini's theorem apply}. Now since $m_2(t) \geq 0$ we have that the the integral of the last quantity must be finite. 
\end{proof}

\smallskip

We now seek to convert this lemma into a slightly more convenient form. To that end, we use the following fact, whose proof is routine.

\smallskip

\begin{lemma} Let $f: \R \rightarrow \R$ be a Borel measurable function and let $A$ be a set with positive Lebesgue measure $\mathcal{L}(A)$. Then
\[ \int_{(x,y) \in A \times A} (f(x) - f(y))^2  = 2 {\mathcal L}(A) \int_{x \in A} \left( f(x) - \frac{1}{\mathcal{L}(A)} \int_{y \in A} f(y) \right)^2  \]
\label{varchange} \end{lemma}

Our next step is to combine the previous two lemmas \jmh{with the assumptions on positive interactions at distance less than $r$} into the following result. \jmh{{\sl Thus the assumptions of Theorem \ref{densityconvergence} will be henceforth made for all results in the remainder of Section \ref{densproof}.}}

\smallskip

\begin{lemma} \label{pointpicking} Let $I$ be an interval of length less than $r$ and let $B(I,t)$ be the set of $\alpha$ such that $x_t(\alpha) \in I$. 
There is a \ao{function $z(t)$} whose range lies in $I$ such that 
\[ \lim_{t \rightarrow \infty} \int_{\aor{\alpha \in} B(I,t)} \left( x_t(\alpha) - \ao{z(t)} \right)^2  = 0 \]
\end{lemma}

\begin{proof}
We first prove that 
\begin{equation}\label{eq:lim_square_I}
\lim_{\aor{t \rightarrow \infty} } \int_{\alpha, \beta \in B(I,t)} (x_t(\alpha) - x_t(\beta))^2 = 0.
\end{equation}

Indeed, suppose that $I = [a,b]$ and let $I' = [a-\epsilon,b+\epsilon]$ where we pick 
$\epsilon$ so that the length of $I'$ is \aor{positive but } less than $r$. Suppose, to obtain a contradiction, that \eqref{eq:lim_square_I} does not hold and thus that \begin{equation} \label{intlowerbound} \int_{\alpha, \beta \in B(I,t_k)} (x_{t_k}(\alpha) - x_{t_k}(\beta))^2   >  \epsilon' \end{equation} \ao{for some $\epsilon' > 0$ and for a sequence of times $t\aor{_k}$ which approaches infinity}. Let $\Delta_{t_k}\subseteq B(I,t_k) \times B(I, t_k)$ be the set of ordered pairs $(\alpha,\beta)$ for which $(x_{t_k}(\alpha) - x_{t_k}(\beta))^2 > \epsilon'/2$. Since $x_t(\alpha) \in [0,1]$ by Lemma \ref{lem:x_[01]_Lips}, it follows from Eq. \eqref{intlowerbound} that $\Delta_{t_k}$ has measure at least $\epsilon'/2$ (else, the integral in Eq. (4.2) would be upper bounded by $\epsilon'$).

Thanks to the Lipschitz continuity of $x$ with respect to time (Lemma \ref{lem:x_[01]_Lips} again), we can then conclude that there exist a sequence of times $t_{k}'>t_k$ with the property that $t_k'-t_k$ is uniformly bounded from below, and such that for all $t\in [t_k,t_k']$ 
and $\alpha, \beta \in \Delta_{t_k}$ there holds (i) $(x_t(\alpha)-x_t(\beta))^2 \geq \epsilon'/4$ and $x_t(\alpha), x_t(\beta) \in I'\supset I$. This implies that
\begin{equation}\label{eq:diverging_int}
\int_t \int_{\alpha, \beta \in B(I',t)} (x_t(\alpha) - x_t(\beta))^2 \aor{=} \infty.
\end{equation}
Recall, however, that  $w(t,\alpha,\beta,x_t(\alpha),x_t(\beta)) \geq \Delta >0$ when $\abs{x_t(\alpha)-x_t(\beta)}<r$, which is always the case if $\alpha,\beta \in B(I',t)$. The divergence in \eqref{eq:diverging_int} implies thus the divergence of 
$$
\int_t \int_{\alpha, \beta \in B(I',t)} w(t,\alpha,\beta,x_t(\alpha),x_t(\beta))(x_t(\alpha) - x_t(\beta))^2,
$$
in contradiction with Lemma \ref{finiteintegrability}. We conclude that \eqref{eq:lim_square_I} must hold. 

To conclude our proof,  set
\[ z(t) = \frac{1}{{\mathcal L}(B(I,t))} \int_{\aor{\alpha \in}B(I,t)} x_t(\alpha) \] and observe that
\[ \int_{\aor{\alpha \in} B(I,t)} \left( x_t(\alpha) - \ao{z(t)} \right)^2  \leq \min \left( {\mathcal L}(B(I,t)), ~  \frac{1}{2 {\mathcal L}(B(I,t))} \int_{\alpha, \beta \in B(I,t)} (x_t(\alpha) - x_t(\beta))^2 \right) \] where we used  Lemma \ref{varchange}. Performing some elementary manipulations on limits, this implies that 
$\lim_{t \rightarrow \infty} \int_{\alpha,\beta \in B(I,t)} (x_t(\alpha) - z(t))^2 = 0.$
\end{proof}

\smallskip

We now show that $z(t)$ in the previous lemma can actually be taken as constant independent of $t$.

\smallskip

\ao{\begin{definition} We call any point $x$ with $\mu_{\infty}({x})=0$ a {\em continuity point}. Note that if $x \notin [0,1]$ then $x$ is automatically a continuity point. An interval $[a,b]$ is a {\em continuity interval} if both $a$ and $b$ are continuity points. Having established Theorem \ref{weakconv}, we know that for any continuity interval $I$, $\mu_t(I) \rightarrow \mu_{\infty}(I)$.  
\end{definition}}
\bigskip

$ $
\bigskip

\ao{\begin{lemma} \label{lem:z_in_or_out}
There exist a \jmh{finite sequence of points}  $z_1, \ldots, z_k$ in $[0,1]$ \jmh{such that}: \begin{enumerate} \item If $I$ is a closed interval 
not containing any of the points  $z_i$, then $\mu_{\infty}(I)=0$. 
\item If $I$ is a closed interval whose interior contains at least one $z_i$, then $\mu_{\infty}(I)>0$. In fact, there exists some $l>0$ such that $\mu_{\infty}(I) \geq l$ for any such $I$.
\end{enumerate} 
\label{limitpts} 
\end{lemma}}

\smallskip

\begin{proof} 
Let $m$ be an integer larger than $1/r$, and let us cover the interval $[0,1]$ with $m$ successive continuity intervals $J_1, \ldots, J_m$ of length strictly less than $r$ (note that the starting point of the first interval can be below zero and the endpoint of the last interval can be above one). 
Since a measure can have at most countably many points which are not continuity points, such a partition always exists. It can for example be obtained by starting from a partition in $m$ intervals of equal length slightly perturbing the endpoints that would not be continuity points.

Applying Lemma \ref{pointpicking}, get the existence of $m$ \ao{functions} $z_t(i)$ such that for each $l=1, \ldots, m$,  \begin{equation} \label{jconc} \lim_{t \rightarrow \infty} \int_{\aor{\alpha \in} B(J_i,t)} (x_t(\alpha) - z_t(i))^2 = 0 \end{equation} Note that by definition ${\mathcal L}(B(J_i,t)) = \mu_t(J_i)$. We have, therefore, the following concentration result implied by \ao{Eq. (\ref{jconc})}: either $\mu_t(J_i)$ approaches $0$ as $t \rightarrow \infty$ or, for any $\epsilon>0$ there is a time after which the measure of the set of agents with values in $J_i \cap [z_i(t)-\epsilon, z_i(t) + \epsilon]$ is at least $(1-\epsilon) \mu_{\infty}(J_i)$. 

We next claim that if $\mu_{\infty}(J_i) > 0$ then $z_t(i)$ converges as $t \rightarrow \infty$. Indeed, if some such $z_t(i)$ had two distinct limit points, then the concentration result of the previous paragraph would immediately contradict the convergence of the measures associated to $x_t$ established in Theorem \ref{weakconv}. Finally, we pick $z_1, \ldots, z_k$ to be the limits of those $z_t(i)$ with $\mu_{\infty}(J_i)>0$. 

\ao{Having defined the points $z_i$, we now turn to the statement of the lemma. Indeed, consider some interval $I$. There are two possibilities considered by this lemma.}

\ao{If the interior of $I$ contains some $z_i$, we can choose $\epsilon$ small enough so that $[z_i - \epsilon, z_i + \epsilon]$ lies strictly inside $I$. We can perturb the endpoints of $I$ to bring them closer to each other to get the interval $I'$ which is a continuity interval still containing $[z_i - \epsilon, z_i + \epsilon]$. As we argued above, eventually there is always a $(1-\epsilon) \mu_{\infty}(J_i)$ mass of agents in a subinterval of $[z_i - \epsilon, z_i + \epsilon]$, where $\mu_{\infty}(J_i)$ is strictly positive.
This implies that eventually $\mu_t(I')$ is positive and bounded from below. Moreover, since $I'$ is a continuity interval, $\mu_t(I')$ converges to $\mu_\infty(I')$, so that $\mu_{\infty}(I')>0$. Moreover, we can further lower bound $\mu_{\infty}(I') \geq l$ where $l$ is the smallest positive $\mu_{\infty}(J_i)$. Since $I' \subset I$, these claims hold for $I$ as well.}

\ao{Suppose now that $I$ does not contain any $z_i$. We may write $I$ as a disjoint union $I = \bigcup_{i=1}^m J_i \cap I.$ Consider each intersection in this union. Either $\mu_{\infty}(J_i)$ has measure zero, in which case so does
$\mu_{\infty}(J_i \cap I)$; or $\mu_{\infty}(J_i) > 0$. In the latter case, the interval $J_i$ will contain an element from the set $\{z_j\}$; say $J_i$ contains $z_i$. We can take a continuity interval $[z_i - \epsilon, z_i + \epsilon]$ of small enough but positive length so that $[z_i - \epsilon, z_i + \epsilon]$ does not intersect with $J_i \cap I$. Moreover, in this case $J_i \setminus ( J_i \cap [z_i - \epsilon, z_i + \epsilon]) $ is either a continuity interval or the union of two continuity intervals, because the endpoints of $J_i$ as well as $z_i - \epsilon, z_i + \epsilon$ are continuity
points by construction. It follows then from Theorem \ref{weakconv} that $\mu_t(J_i \setminus ( J_i \cap [z_i - \epsilon, z_i + \epsilon])$, which is no greater than  $\epsilon \mu_{\infty}(J_i)$ for sufficiently large $t$, converges to  $\mu_\infty(J_i \setminus ( J_i \cap [z_i - \epsilon, z_i + \epsilon])$, which is thus smaller than or equal to $\epsilon \mu_{\infty}(J_i)$. \aorev{The same is therefore true of $\mu_{\infty}(J_i \cap I)$}.
Since $\epsilon$ 
could be chosen arbitrarily small, we obtain that $\mu_{\infty}(J_i \cap I)=0$ in this case as well. Finally, $\mu_{\infty}(I)=0$.}
\end{proof} 

\smallskip

\ao{The previous lemma allows us to explicitly characterize the set of continuity points of $\mu_{\infty}$.}

\smallskip

\ao{\begin{corollary} The set $[0,1] \setminus \{z_i\}$ is the set of continuity points of the measure $\mu_{\infty}$. \label{conti}
\end{corollary}}

\smallskip

\ao{\begin{proof} Any point $x \notin \{z_j\}$ has the property that we can take an interval $I$ containing it with $\mu_{\infty}(I)=0$ by Lemma \ref{lem:z_in_or_out} so that $x$ is a continuity point of $\mu_{\infty}$. Conversely, for each $z_i \in J_i$, take a decreasing sequence of intervals $I_k$ containing it whose
length approaches zero. For each such interval, by Lemma \ref{lem:z_in_or_out}  we have $\mu_{\infty}(I_k) \geq l$ for some $l>0$, and consequently $\mu_{\infty}(\{z\}) > 0$. 
\end{proof}}

\smallskip

\ao{It is not too hard to see that the points $z_i$ cannot be too close together; this is formally stated in the following lemma.} 

\smallskip

\ao{\begin{lemma} For all $i,j$, $|z_i - z_j| \geq r$. 
\end{lemma}}

\smallskip

\ao{\begin{proof} If $0<|z_i - z_j| < r$ for some $i$,$j$, then we can take a small enough interval $I_i$ containing $z_i$ and a small enough interval
$I_j$ containing $z_j$ such the distance between any point in $I_i$ and any point in $I_j$ is in $[|z_i - z_j|/2, r]$. By Lemma \ref{limitpts} and Corollary \ref{conti}, there exists a time $t'$ such that for all $t \geq t'$ we have $\mu_t(I_i) \geq l/2, \mu_t(I_j) \geq l/2$. This means that 
$$\int_{\aor{\alpha, \beta \in} [0,1]^2} w(t, \alpha, \beta, x_t(\alpha), x_t(\beta)) (x_t(\beta) - x_t(\alpha))^2 \geq \Delta \left( \frac{l}{2} \right)^2 \left(\frac{|z_i - z_j|}{2} \right)^2$$ for all $t \geq t'$, which contradicts
Lemma \ref{finiteintegrability}.
\end{proof} }

\smallskip

\ao{The proof of Theorem \ref{densityconvergence} will require us to apply Lemma \ref{limitpts} to various intervals. It will be convenient to have the following definition, which 
implicitly defines a set of intervals around the points $z_i$.} 

\smallskip

\ao{\begin{definition} A positive real number $\epsilon$ will be called {\em feasible} if \begin{enumerate} \item  $\epsilon<r/2$.  \item The intervals $[z_i - \epsilon, z_i + \epsilon]$ do not intersect. \item For all $i$,  $z_i + \epsilon$ is 
at most $2$ while $z_i - \epsilon$ is at least $-1$.
\end{enumerate}
\end{definition}}

\smallskip

\ao{Clearly, all small enough $\epsilon$ are feasible. }

\smallskip

\ao{\begin{definition} Given a feasible $\epsilon$, we define 
 $I_{\epsilon} = [-1,2] \setminus \bigcup_i [z_i - \epsilon, z_i + \epsilon]$. Furthermore, we define ${\cal B}_{\epsilon}$ be the set of 
agents $\beta$ such that $x_t(\beta) \in I_{\epsilon}$ for all $t$ large enough. 
\end{definition}}

\smallskip

\ao{\begin{corollary}  For any feasible $\epsilon$, ${\cal L}({\cal B}_{\epsilon})=0$. \label{bmeasure}
\end{corollary}}

\smallskip

\ao{\begin{proof} Pick any \aor{increasing} sequence $t_k$ approaching positive infinity and let ${\cal B}_{\epsilon}(t_k)$ be the set of agents whose values lie in $I_{\epsilon}$ for all $t \geq t_k$.Then \aor{the sequence} ${\cal B}_{\epsilon}(t_k), \aor{k=1,2, \ldots}$ is a nondecreasing set sequence \aor{(due to the fact that $t_k$ is increasing)} that approaches ${\cal B}_{\epsilon}$, and so ${\cal B}_{\epsilon}$ is measurable and  ${\cal L}( {\cal B}_{\epsilon}(t_k )) \rightarrow {\cal L}({\cal B}_{\epsilon})$.  However ${\cal L}({\cal B}_{\epsilon}(t_k)) \leq \mu_{t_k}(I_{\epsilon})$ by definition of $\mu_t$ and $\mu_{t_k}(I_{\epsilon}) \rightarrow 0$ by Lemma \ref{limitpts} and Theorem \ref{weakconv}. Consequently, ${\cal L}({\cal B}_{\epsilon}) = 0$.  
\end{proof}}
\smallskip

\ao{The last corollary is a key step towards the proof of Theorem \ref{densityconvergence}: it states that agents which stay bounded
away from any $z_i$ after a time form a measure zero set. However, to prove Theorem \ref{densityconvergence} we will need something additional to rule out
the possibility that a positive measure set of agents have multiple limit points which include $z_i$.

To that end, we argue that only a set of measure zero of agents cross any small intervals close to a point $z_i$ infinitely often. This is formalized
in the next series of definitions and lemmas.}

\smallskip

\begin{lemma} \ao{Let $I$ be a closed interval contained in some $(z_i, z_i + \epsilon)$ or $(z_i - \epsilon, z_i)$ for some feasible $\epsilon$.} \jmh{Then} $\int_{0}^{\infty} \mu_t(I) < \infty$.
\end{lemma} 

\smallskip

\begin{proof} \ao{Without loss of generality, suppose $I=[a,b]$ is contained in $(z_i, z_i + \epsilon)$; the other case proceeds similarly. Now by} Lemma \ref{finiteintegrability} \ao{we have that}
\begin{equation} \int_0^{\infty} \int_{\aor{\alpha, \beta \in} [0,1]^2} w(t, \alpha, \beta, x_t(\alpha), x_t(\beta)) (x_t(\beta) - x_t(\alpha))^2  ~dt < \infty 
\label{eq:restate_finite_integrability}
\end{equation} 
\ao{Let $\delta = a - z_i$ and pick $I'$ to be any interval contained within $(z_i - \epsilon, a-\delta/2)$ which contains $z_i$ and  let $q_i = \mu_{\infty}(I')$. By Lemma \ref{limitpts}, we have that $q_i>0$ and by Corollary \ref{conti}, $I'$ is a continuity interval. Therefore, for large enough $t$, we will have 
\[ \int_{\aor{\alpha, \beta \in} [0,1]^2} w(t, \alpha, \beta, x_t(\alpha), x_t(\beta)) (x_t(\beta) - x_t(\alpha))^2  \geq \Delta \left( \frac{\delta}{2} \right)^2  \left( \frac{q_i}{2} \right) \mu_t(I) \]}
which immediately implies \jmh{that the statement of the current lemma holds true, as \eqref{eq:restate_finite_integrability} could otherwise not be satisfied}.
\end{proof} 

\smallskip

\ao{\begin{definition} Given an interval $I$ and $\alpha \in [0,1]$, define $T(\alpha,I)$ to be the set of times $t$ such that $x_t(\alpha) \in I$. Moreover, define ${\mathcal A}(I)$ to be the set of all $\alpha$ such that ${\mathcal L}(T(\alpha,I)) = \infty$, i.e. those spending an infinite amount of time in the interval $I$. Because $x_t(\alpha)$ is continuous in $t$ for fixed $\alpha$ and measurable in $\alpha$ for fixed $t$, it is jointly Borel measurable in
$t$ and $\alpha$, and so these sets are measurable. \label{mrk}
\end{definition}}

\smallskip

\begin{lemma} \ao{Let $I$ be a closed  interval  contained in some $(z_i, z_i + \epsilon)$ or $(z_i - \epsilon, z_i)$ for some feasible $\epsilon$.} Then ${\mathcal A}(I)$ has Lebesgue measure $0$. \label{intervallemma} 
\end{lemma}

\smallskip 

\begin{proof} By the previous lemma,  $\int_{0}^{\infty} \mu_t(I) ~dt < \infty$. We can rewrite this as $\int_{0}^{\infty} \int_0^1 1_{x_t(\alpha) \in I} ~d \alpha ~dt   < \infty$. Note that because $x_t(\alpha)$ is jointly measurable in $t$ and $\alpha$, the function $1_{x_t(\alpha) \in I}$ is measurable and the
above expression makes sense. Since the function $1_{x_t(\alpha) \in I}$ is nonnegative,  by Tonelli's theorem
we can interchange the order of integration to obtain $\int_0^1 \int_{0}^{\infty} 1_{x_t(\alpha) \in I} ~dt ~d \alpha < \infty$, 
or  $\int_0^1 \aorev{{\mathcal L }} \left(T(\alpha, I)  \right)~d \alpha < \infty$,
which implies that $\aorev{{\mathcal L}} \left( T(\alpha,I) \right)$ can be infinite only on a set of $\alpha$'s of measure $0$. 
\end{proof} 

\smallskip

\ao{\begin{corollary} Let $\epsilon$ be feasible and let $I$ be a closed interval contained in some $(z_i, z_i+\epsilon)$ or $(z_i-\epsilon, z_i)$. Then the set of agents $\alpha$ having the property that, for some sequence of times $t_k \rightarrow +\infty$ we have that $x_{t_k}(\alpha) \in I$, is a subset of a measure-zero set. \label{manyvisits}
\end{corollary} }

\ao{\begin{proof} Let $I'$ be a closed interval which contains $I$ and is contained in the same $(z_i, z_i + \epsilon)$ or $(z_i-\epsilon, z_i)$. Since 
$x_t(\alpha)$ is Lipschitz continuous by Lemma \ref{lem:x_[01]_Lips}, 
we have that if agent $\alpha$ visits $I$ \aorev{at some sequence of times approaching infinity} then $\alpha \in {\cal A}(I')$, because the time needed to go to/from $I$ from/to outside $I'$ is bounded away from 0. Lemma \ref{intervallemma} then
implies the current corollary. 
\end{proof}}

\smallskip

\ao{We are now ready to prove Theorem \ref{densityconvergence}. Our proof will only rely on the last Corollary \ref{manyvisits} as well as Lemma \ref{bmeasure} proved earlier. As we argue next, these two facts immediately imply that only a set of measure zero fo agents does not 
converge to one of the $z_i$. }

\bigskip

\ao{\begin{proof}[Proof of Theorem \ref{densityconvergence}] Without loss of generality, let us assume that $z_i$ are sorted in increasing order, i.e., \aorev{$z_1 < z_2 < \cdots < z_p$}. Suppose agent $\alpha$ does not converge to \aorev{an element of the set $\{z_1, \ldots, z_p\}$}, and consider the set of limit points of $x_t(\alpha)$ \aorev{(i.e, the set of points which are limits of $x_{t_k}(\alpha)$ for some sequence of times $t_k \rightarrow +\infty$)}. There are several possibilities. 

\begin{enumerate} 
\item If $x_t(\alpha)$ has a limit point which is $z_i$ and another limit point which is strictly larger than this $z_i$, then it crosses some closed interval
$I$ with rational endpoints contained in some $(z_i, z_i+\epsilon)$ for a feasible $\epsilon$ \aorev{during a sequence of times $t_k \rightarrow +\infty$.}
\item If $x_t(\alpha)$ has a limit point which is $z_i$ and another limit point which is strictly less than this $z_i$, then it crosses some closed interval $I$ with rational endpoints
contained in some $(z_i-\epsilon, z_i)$ for a feasible $\epsilon$ \aorev{during a sequence of times $t_k \rightarrow +\infty$.}
\item If $x_t(\alpha)$ has a limit point which is strictly less than some $z_i$ and a limit point which is strictly greater than \aorev{the same} $z_i$, 
then it crosses some closed interval $I$ with rational endpoints
contained in some $(z_i-\epsilon, z_i)$ for a feasible $\epsilon$ \aorev{during a sequence of times $t_k \rightarrow +\infty$.}
\item If all the limit points of $x_t(\alpha)$ are contained in some $(z_i, z_{i+1})$ then {$x_t(\alpha)$} belongs to some $\aor{{\cal B}}_{1/m}$ for a large enough $m$, i.e. $x_t(\alpha)$ remains at a distance larger than $1/m$ from all $z_i$ after a certain time.
\item \aorev{Finally, if all the limit points of $x_t(\alpha)$ are strictly below $z_1$ or strictly above $z_p$, then just as in the previous item $x_t(\alpha)$ belongs to some ${\cal B}_{1/m}$.}
\end{enumerate}

Now Lemma \ref{bmeasure} tells us that ${\cal L}({\cal B}_{\epsilon}) = 0$ for all feasible $\epsilon$. Taking a countable union over $\epsilon = 1/m$ with  \aorev{$m$ integer} gives us that the set of agents which satisfy (1) has measure zero. The same argument coupled with Lemma \ref{manyvisits} shows the set of agents that satisfy the conditions of each
of the items $(2)-(4)$ is a subset of a set measure zero. To summarize, if $x_t(\alpha)$ does not converge to some $z_i$, we have that it lies inside a union of four measure zero sets. 

It remains to argue that the set of $\alpha$ such that $x_t(\alpha)$ does not converge to one of the $z_i$ is measurable. We know 
the function $\lim \sup_{n \rightarrow \infty} x_{n/k}(\alpha)$ is measurable for any fixed positive integer $k$; consequently, 
$F(\alpha) = \lim \sup_{t \geq 0} x_t(\alpha)$ is measurable since $|\dot{x}_t(\alpha) | \leq W$ implies that $F(\alpha) = 
\lim_{k \rightarrow \infty} \lim \sup_{n \rightarrow \infty} x_{n/2^k}(\alpha)$. Similarly, $G(\alpha) = \lim \inf_t x_t(\alpha)$ is measurable. Then the 
set of $\alpha$ such that $x_t(\alpha)$ does not converge to one of the points $z_i$ is $[0,1] \setminus  \bigcup_{i} \left(F^{-1} (z_i) \cap G^{-1}(z_i) \right)$
so it is measurable.

\end{proof} }

\subsection{Proof of Theorem \ref{thm:Lips->conv}}

We will show that under the assumptions of Theorem \ref{thm:Lips->conv}, the order of the opinions is preserved, in the sense that the difference of opinions between two agents never changes sign, and that together with convergence in distribution this is sufficient to prove convergence for almost all agents. 

\smallskip

Formally, consider the functions $x_t(\alpha):[0,1]\to \Re$ for every $t\in \Re^+$. 
We say that the \emph{order of $x_t$ is preserved} if $x_t(\beta) > x_t(\alpha)$ for some $t$ implies $x_s(\beta)> x_s(\alpha)$ for all $s$. By contraposition, it is equivalent to requiring that 
$x_t(\beta) \leq x_t(\alpha)$ if an only if $x_s(\beta) \leq  x_s(\alpha)$ for all $s$, and implies thus that $x_t(\beta) = x_t(\alpha)$ for some $t$ if and only if $x_s(\beta) = x_s(\alpha)$ for all $s$. 

The following lemma shows that convergence in distribution implies convergence for almost all $\alpha$ when the order is preserved. Intuitively, the result is quite clear: the measure $\mu_t$ contains exactly the same information as $x_t$ up to possible \quotes{relabelling} or \quotes{switching} of the agents. But when the order is preserved, no such switching can take place, so that convergence in distribution implies convergence for almost every $\alpha$. A formal proof is presented in Appendix \ref{sec:proof_lemma_order}.

\smallskip
\begin{lemma}\label{lem:order+weak->a.e.}
Suppose that $x_t$ converges in distribution: 
There exists a measure $\mu_{\infty}$ on $[0,1]$ such that $\mu_t$ approaches $\mu_{\infty}$ in distribution, where $\mu_t$ is defined as in Eq. (\ref{mudef}).\\
If  $x_t(\alpha) \in [0,1]$ holds for all $t,\alpha$, and if the order of $x_t$ is preserved, then $x_t(\alpha)$ converges for almost all $\alpha$. 
\end{lemma}
\smallskip

Intuitively, the order should always be preserved if the interactions are entirely determined by the agents' positions and by time. Indeed, suppose that $x_t(\beta)>x_t(\alpha)$ for some $t$, and $x_s(\beta)<x_s(\alpha)$ for some $s>t$. The continuity of $x$ implies the existence of a time $s'\in(t,s)$ at which the agents have the same position $x_{s'}(\beta) = x_{s'}(\alpha)$. But since the interactions are entirely determined by the positions, the opinions of $\alpha$ and $\beta$ will at that point be subject to exactly the same attractions, and they thus should remain equal forever. As a result, we could never have  $x_s(\beta)<x_s(\alpha)$ after $x_t(\beta)>x_t(\alpha)$. Similar intuitive arguments can be used to suggest that $x_t(\beta) = x_t(\alpha)$ if and only if $x_s(\beta) = x_s(\alpha)$ for all $s$ and that the order of $x$ is thus preserved.

These intuitive arguments are however not formally valid, as they implicitly rely on the uniqueness of the solution to the equation (\ref{mainiteration}) describing the evolution of the opinions of $\alpha$ and $\beta$ with initial conditions $x_{s'}(\beta) = x_{s'}(\alpha)$ at time $s'$, and this uniqueness is in general not guaranteed. Different issues related to uniqueness of solutions have been reported even for very simple models, for example in \cite{bht10}.

Nevertheless, preservation of the order can be established under an additional smoothness assumption that guarantees that two agents with different positions never reach a same point in finite time.

\smallskip

\begin{lemma}\label{lem:lips->order}
Suppose that the interaction weights $w(.,.,.,.,.)$ are symmetric, nonnegative, bounded by some $ W\in \R$, and only depend on time and the positions:
$$w(t,x_t(\alpha),x_t(\beta),\alpha,\beta)  = \tw(t,x_t(\alpha),x_t(\beta)).$$
If $\tw$ is 
Lipschitz continuous with respect to $x$, i.e. there exists a $L$ such that $|\tw(t,x,y) - \tw(t,x,z)| \leq L \abs{y-z}|$ for all $t,x,y,z$,
then the order of any solution $x_t$ of the system \eqref{mainiteration} is preserved.
\end{lemma}

\begin{proof}
Consider two arbitrary agents $\alpha,\beta$, and observe that
\begin{align*}
\dot x_t(\beta) -&\dot x_t(\alpha)=\\
 &\int_0^1 \tw(t,x_t(\beta),x_t(\gamma)) (x_t(\gamma) - x_t(\beta)) d\gamma -\tw(t,x_t(\alpha),x_t(\gamma)) (x_t(\gamma) - x_t(\alpha)) d\gamma \\
=& \int_0^1 \prt{ \tw(t,x_t(\beta),x_t(\gamma))- \tw(t,x_t(\alpha),x_t(\gamma))}\prt{x_t(\gamma)-x_t(\alpha)}d\gamma  \\ &+ \int_0^1  \tw(t,x_t(\beta),x_t(\gamma)) \prt{x_t(\alpha)-x_t(\beta)}d\gamma.
\end{align*}
Using the Lipschitz constant $L$ and the upper bound $\tw \leq W$ on this inequality leads to   
\begin{equation}\label{eq:bound_order_almost_final}
\abs{\dot x_t(\beta) -\dot x_t(\alpha) } \leq
\int_0^1 L  \abs{x_t(\beta)-x_t(\alpha)}\abs{x_t(\gamma)-x_t(\alpha)}d\gamma  + \int_0^1   W \abs{x_t(\alpha)-x_t(\beta)}d\gamma 
\end{equation}
Remember that $x_0(\alpha)$ is assumed to lie in $[0,1]$ throughout this paper. Moreover, we have seen in Lemma \ref{lem:x_[01]_Lips} how this implies that $x_t(\gamma)\in [0,1]$ for all $t\geq 0, \gamma\in [0,1]$.
As a result, there holds $\int_0^1 \abs{x_t(\gamma)-x_t(\alpha)}d\gamma \leq 1$, which reintroduced in  \eqref{eq:bound_order_almost_final} yields
$$
\abs{\dot x_t(\beta) -\dot x_t(\alpha) } \leq  \prt{ L + W} \abs{x_t(\alpha)-x_t(\beta)}
$$
This last bound implies that $x_t(\beta)-x_t(\alpha)$ decreases or increases at most exponentially fast, with a rate bounded in absolute value by $ L + W$. Therefore, if $x_t(\beta) > x_t(\alpha)$ for some $t$, then $x_s(\beta) > x_s(\alpha)$ for all $s$ (The fact is obvious for $s\geq t$, and can easily be seen for $s<t$ by reversing the time). This implication being true for any pair of agents $\alpha,\beta$, the order of $x_t$ is preserved.
\end{proof}

\smallskip

Note that the result still holds if the bounds $W$ and $L$ depend on time, provided that 
$\int_{s=0}^tL(s)ds<\infty$ and $\int_{s=0}^t W(s)ds <\infty$. The proof is exactly the same.

\smallskip

Since Lemma \ref{lem:x_[01]_Lips} states that $x_t(\alpha) \in [0,1]$ for all $\alpha,t$, the result of Theorem \ref{thm:Lips->conv} follows then directly from  from the convergence in distribution of $x_t$, guaranteed by Theorem \ref{weakconv}, and the combination of Lemma \ref{lem:lips->order} with Lemma \ref{lem:order+weak->a.e.}.

\section{Conclusion\label{conclusion}} 

Our goal in this paper has been to analyze the asymptotic behavior of 
opinion dynamics. We have been able to resolve several questions implicit in the previous literature
on the subject. In Theorem \ref{weakconv} we proved that symmetry alone appears to suffice for the convergence in distribution of such systems.
Moreover, we showed in Theorems \ref{counterexample}, \ref{densityconvergence} and \ref{thm:Lips->conv} that while these systems converge in the sense of distributions, convergence for almost all agents is not automatic but rather crucially depends on additional assumptions, in sharp opposition with the results obtained for systems with finitely many agents \cite{HT:2013}. 


Our motivation for studying these systems has been in the similarity they share with other multi-agent systems, namely the presence of a 
nonlinearity due to a time-varying update rule. 
in the analysis of other multi-agent systems. 
We note that Theorem \ref{weakconv} was proved
without precise Lyapunov estimates on the decay by instead relying on a large class of Lyapunov functions coupled with an appeal to results
concerning the Haussdorff moment problem; to our knowledge, this is a \ao{completely} new approach.

We conclude with an open question. The contrast between Theorem \ref{counterexample} on the one hand and Theorems \ref{densityconvergence} and \ref{thm:Lips->conv} on the other hand
leads one to wonder whether a precise characterization of the settings in which convergence occurs for almost all agents is possible.


\section*{Acknowledgements} 
The authors wish to thank Francesca Ceragioli and Paolo Frasca for their help on the relevant assumptions needed
on the solutions $x$.

\appendix
\section{A Brief Note on Existence and Uniqueness Issues}

In general, analyzing the conditions under which existence and uniqueness holds for the integro-differential equation of Eq. (1.1) appears to be challenging, and is out of scope of the present paper. Nevertheless, we would like to prove that existence and uniqueness does hold in a large number of interesting cases. We now proceed to state a theorem to this effect.

We first require some definitions. As in the body of the paper, we will use $x_t(\alpha)$ or $y_t(\alpha)$ to denote functions from $[0, \infty) \times [0,1]$ to $\mathbb{R}$. For such  functions, we define the truncated infinity norm in the usual way, $\jmhr{||y||_{\infty,T}} = \max_{t \in [0,T], \alpha \in [0,1]} |y_t(\alpha)|$.

Our existence and uniqueness theorem, stated next, tells us that subject to some continuity and Lipschitz assumptions on the function $w(\cdot, \cdot, \cdot, \cdot, \cdot)$ in Eq. (1.1) there exists exactly one solution of Eq. (1.1) satisfying  the conditions we have assumed in the body of the paper. 

\smallskip

\begin{theorem} Suppose $w(\cdot, \cdot, \cdot, \cdot, \cdot)$ is a jointly measurable function of five arguments which is continuous in the first, fourth, and fifth argument and 
Lipschitz in each of the last two arguments with Lipschitz constant $L$. Furthermore, suppose $w(\cdot, \cdot, \cdot, \cdot, \cdot)$ also satisfies the bound  $|w(\cdot, \cdot, \cdot, \cdot, \cdot)| \leq W < +\infty$.

\smallskip

Then, given any measurable \jmhr{initial condition} function $\tilde x_0: [0,1] \rightarrow [0,1]$, there exists a unique function $x_t(\alpha)$ having the following four properties: 
\begin{enumerate}
\item $x_t(\alpha)$ satisfies Eq. (1.1) for all $t \in [0,\infty)$ and $\alpha \in [0,1]$.
\item $x_t(\alpha)$ has the property that $x_0(\alpha) = \tilde x_0(\alpha)$ for all $\alpha \in [0,1]$. 
 \item $x_t(\alpha)$ is measurable in $\alpha$ for any fixed $t \geq 0$ and continuously differentiable in $t$ for any fixed $\alpha \in [0,1]$.

\item $ $\\
\vspace{-.9cm}
\begin{equation} \label{trunc} 
||x||_{\infty,T} < \infty 
\mbox{ for all } T \geq 0. 
\end{equation}
\end{enumerate}
\label{app-thm}
\end{theorem}
We prove this theorem here in order to demonstrate that the basic object of study of this paper, namely 
solutions of Eq. (\ref{mainiteration}) which do not explode in finite time, exist for a large class of opinion dynamic
models (and are unique).

Nevertheless, 
the assumptions under which we are able to prove the above theorem are considerably more restrictive than the assumptions under
which we can prove Theorems 1.1-1.4 on properties of solutions that do exist. Besides the continuity and Lipschitz assumptions on $w(\cdot, \cdot, \cdot, \cdot, \cdot)$, note that this theorem
discusses functions $x_t(\alpha)$ which are continuously differentiable and for which Eq. (1.1) holds everywhere, whereas Theorems 1.1-1.4 only require $x_t(\alpha)$ to be absolutely 
continuous and Eq. (1.1) to hold almost everywhere. 
Establishing existence and uniqueness of Eq. (1.1) in more general settings is therefore an open problem.

We now turn to the proof, which is a variation of the usual Picard iteration arguments. Our first lemma recasts the problem as an integral equation. This is slightly more delicate than usual, as all our integrals are Lebesgue integrals and we must therefore exercise some care in applying the fundamental theorems of calculus. 
\smallskip
\begin{lemma} Consider the integral equation 
\begin{equation} \label{intform} x_t(\alpha) = \tilde x_0(\alpha) + \int_0^t \int_0^1 w(u, \alpha, \beta, x_u(\alpha), x_u(\beta)) (x_u(\beta) - x_u(\alpha) ~ d \beta ~ du 
\end{equation} 
\jmhr{Let $x_t(\alpha)$ be continuous in $t$ for any fixed $\alpha \in [0,1]$, measurable in $\alpha$ for any fixed $t \in [0,\infty)$, satisfying the local boundedness condition (\ref{trunc}), and the initial condition $x_0(\alpha)=\tilde x_0(\alpha)$.

Then, $x_t$ is a solution of Eq. \eqref{intform} if and only if it is continuously differentiable in $t$ for any fixed $\alpha \in [0,1]$ and satisfies Eq. (1.1) for all $t \in [0, \infty), \alpha \in [0,1]$.}
\label{lemma-int}
\end{lemma}
\smallskip

\begin{proof} 
We begin by supposing that  $x_t(\alpha)$ \jmhr{is a solution of Eq. (\ref{intform}), and first prove} 
that the integrand \begin{equation} \label{integrand} \int_0^1 w(u, \alpha, \beta, x_u(\alpha), x_u(\beta)) (x_u(\beta) - x_u(\alpha) ~ d \beta \end{equation} 
\jmhr{of \eqref{intform}} is \jmhr{in that case} a continuous function of $u$ for fixed $\alpha$. 
Indeed, fix a time $u'$ and a neighborhood of $u$. Since $|w(\cdot, \cdot, \cdot, \cdot, \cdot)| \leq W$, and since $x$ satisfies \eqref{trunc}, the expression inside the integral in \eqref{integrand} can be uniformly bounded on that neighborhood by some constant $C$, which can also be seen as a (constant) measurable function of $\beta$ defined  on $[0,1]$. 
Moreover, the expression inside the integral in \eqref{integrand} is continuous with respect to $u$, because $x_u(\alpha)$ is continuous with respect to $u$,  and $w$ is continuous with respect to its first, fourth and fifth arguments. 
The continuity of the integrand \eqref{integrand} follows then from the dominated convergence theorem.
We can thus \jmhr{now} apply the first fundamental theorem of calculus (formally, Theorem 1.6.9 in \cite{tao}), and we have that $x_t(\alpha)$ satisfies Eq. (1.1). 
Moreover, the time derivative of $x_t(\alpha)$ is precisely the expression of Eq. (\ref{integrand}) which we have already \jmhr{shown to be} continuous in $t$ for fixed $\alpha$; we conclude that  $x_t(\alpha)$ is continuously differentiable in $t$ for fixed $\alpha$. This proves the \jmhr{first implication of the lemma}.  

Conversely, suppose now $x_t(\alpha)$ is a solution of Eq. \eqref{mainiteration} which is continuously differentiable in $t$ for fixed $\alpha$. Since $\dot{x}_t(\alpha)$ is continuous in $t$ for fixed $\alpha$, we can apply the second fundamental theorem of calculus (formally, Theorem 1.6.7 in \cite{tao}) to get that $x_t(\alpha)$ satisfies Eq. (\ref{intform}). 
\end{proof}

\medskip

\begin{proof}[Proof of Theorem \ref{app-thm}] By Lemma \ref{lemma-int}, we are looking to establish existence and uniqueness for solutions of Eq. (\ref{intform}) which are continuous in $t$ for fixed $\alpha$, measurable in $\alpha$ for fixed $t$, and satisfy Eq. (\ref{trunc}). Fix $b$ such that $0< b < \min(\frac{1}{4W}, \frac{1}{2(W+4L)} )$, where recall that $L$ is the Lipschitz constant of $w(\cdot, \cdot, \cdot, \cdot, \cdot)$ in each of the last two arguments. \jmhr{We begin by showing that \ref{intform} admits a unique bounded solution on $[0,b]$.}

Define the function \jmhr{$q:[0,b] \times [0,1]\to [0,1]$ by  $q_t(\alpha) = \tilde x_0(\alpha)$}. Let \jmhr{then} ${\cal F}_0$ be the set of functions $y_t(\alpha)$
which map $[0,b] \times [0,1]$ into $\R$ \jmhr{and that} (i) satisfy $\jmhr{||y ||_{\infty,b}\leq 2}$ 
 (ii) are continuous in $t$ for each fixed $\alpha \in [0,1]$ (iii) are measurable in $\alpha$ for each fixed $t \geq 0$. \jmhr{Standard arguments show that ${\cal F}_0$ is a Banach space for the distance induced by the norm $||.||_{\infty,b}$}.

Next, define the operator $P$ on ${\cal F}_0$ by 
\begin{equation} \label{eq:def_integral_op_P}
[Py]_t(\alpha) = \jmhr{\tilde x_0}(\alpha) + \int_0^t \int_0^1 w(u, \alpha, \beta, y_u(\alpha), y_u(\beta)) (y_u(\beta) - y_u(\alpha) ~ d \beta ~ du. 
\end{equation}

\jmhr{Observe that $x$ is a solution of \eqref{intform} (on $[0,b]$) if and only if a it is a fixed point of $P$, $Px=x$. We now show that $P$ admits a unique fixed point. In this purpose, we establish that $P$ maps ${\cal F}_0$ into ${\cal F}_0$ and is a contraction mapping on ${\cal F}_0$. }

\jmhr{Remembering that $||\tilde x_0||_\infty \leq 1$ and that $||y||_{\infty,b} \leq 2$ holds for all $y\in {\cal F}_0$, we obtain from \eqref{eq:def_integral_op_P} that
\[ ||Py||_{\infty,b} \leq ||\tilde x_0||_\infty+    2 b W ||y||_{\infty,b} \leq 1 + 4 b W  \leq 2, \] where the last inequality follows from the definition of $b$. $Py$ satisfies thus condition (i) in the definition of ${\cal F}_0$.}
For condition (ii) on the continuity in $t$ of $Py$, observe that 
\[ \left|  [Py]_t(\alpha) - [Py]_s(\alpha) \right| \leq 2 |t-s| W ||y||_{\infty,b} \leq 4 W |t-s|. \] \jmhr{Finally, it is a consequence of Fubini's theorem that $Py$ is measurable in $\alpha$ for fixed $t$, and satisfies thus condition (iii).} Thus we have shown that $P$ maps ${\cal F}_0$ into ${\cal F}_0$.

\jmhr{Let us now prove that $P$ is a contraction on ${\cal F}_0$. We can rewrite} $[P y]_t(\alpha) - [P z]_t(\alpha)$ as 
\begin{small}
\begin{eqnarray*} 
&  & \int_0^t \int_0^1 w(u, \alpha, \beta, y_u(\alpha), y_u(\beta)) (y_u(\beta) -y_u(\alpha)) - w(u,\alpha, \beta, z_u(\alpha), z_u(\beta)) (z_u(\beta)-z_u(\alpha)) ~d \beta ~d u \\
& = &\int_0^t \int_0^1 \{ w(u, \alpha, \beta, y_u(\alpha), y_u(\beta)) \left(  (y_u(\beta) -y_u(\alpha))  -(z_u(\beta) -z_u(\alpha)) \right) \\
& & \hspace{.5cm} + \left(w(u, \alpha, \beta, y_u(\alpha), y_u(\beta)) - w(u,\alpha, \beta, z_u(\alpha), z_u(\beta)) \right) (z_u(\beta)-z_u(\alpha))  \} ~d \beta ~d u
\end{eqnarray*}
\end{small}
\jmhr{Using the uniform bound $W$ on $w$ and its Lipschitz continuity with respect to the fourth and fifth argument, we obtain then that for every $\alpha$ and $t\in [0,b]$, there holds}
\begin{small}
\begin{eqnarray*} 
|[P y]_t(\alpha) - [P z]_t(\alpha)|
& \leq & \int_0^t \int_0^1 \{ W (| y_u(\alpha) - z_u(\alpha) | + | y_u(\beta) - z_u(\beta) |) \\ &&+ (L |y_u(\beta) - z_u(\beta)| + L| y_u(\alpha) - z_u(\alpha) |)(|z_u(\beta)-z_u(\alpha)|)\}~ d \beta ~ d u\\
& \leq & 2b(W + 4L) ||y-z||_{\infty,b},
\end{eqnarray*}\end{small}
where we have used the fact that $|z_u(\beta)-z_u(\alpha)| \leq 2||z||_{\infty,b}\leq 4$ for every $z\in {\cal F}_0$ and $u\in [0,b]$. 
The definition of $b$ implies then that $P$ is a contraction from ${\cal F}_0$ to ${\cal F}_0$. \jmhr{Banach's fixed point theorem implies then that $P$ admits a unique fixed point in ${\cal F}_0$, which means  \eqref{intform} admits a unique solution in ${\cal F}_0$ on $[0,b]$. 
Since $b$ does not depend on time nor on the initial condition $\tilde x_0$, our result also proves the existence and uniqueness of a solution in ${\cal F}_{t^*}$ over $[t^*,t^*+b]$ for any $t^*$ and \quotes{initial} condition $\tilde x_{t^*}$, where ${\cal F}_{t^*}$ is defined analogously to ${\cal F}_0$. 
Repeatingly applying our argument, we can then construct a solution $x$ over $t\in [0,\infty)$ satisfying conditions 1-4 of the Theorem.}

\jmhr{To prove the uniqueness, suppose that  Eq.  (\ref{intform}) admits a solution $z$ with bounded truncated infinity norm that is different from the solution $x$ that we have constructed. Let $t^*= \sup\{t\geq 0: x_t = z_t\}$, where that latter set is non-empty because $z_0=x_0=\tilde x_0$. }
\jmhr{Since $x$ and $z$ are continuous with respect to $t$, there holds $x_{t^*}=z_{t^*}$. Moreover, such a $z$ would also be a solution of \eqref{mainiteration}, and it follows then from Lemma \ref{lem:x_[01]_Lips} (which is valid for any solution that does not explode in finite time) that $z_t(\alpha)\in [0,1]$ for every $\alpha$ and $t$.}
In particular, $||z_t||_\infty\leq 2$ for every $t\in [t^*,t^*+b]$, and $z$ would thus be in ${\cal F}_{t^*}$. However, our local existence and uniqueness argument applied to $t^*$ and $x_{t^*}$ shows that $x$ is the unique   solution of Eq.  (\ref{intform}) over $[t^*,t^*+b]$ that is in ${\cal F}_{t^*}$ and equal to $x_{t^*}$ at time $t^*$,
 so that $x_t=z_t$ should hold for all $t\in [t^*,t^*+b]$, in contradiction with our assumption.
\end{proof}

\section{Proof of Lemma \ref{lem:order+weak->a.e.}}
\label{sec:proof_lemma_order}

The proof relies on the following idea: if an agent opinion is \quotes{surrounded} by a positive mass of agent opinion (as is the case for almost all of them) and does not converge, its repeated displacement will result is repeated displacements of that mass surrounding it, which forbids the  convergence in distribution. We will need the following technical lemma.

\smallskip
\providecommand{\Leb}{\mathcal{L}}

\begin{lemma}\label{lem:lower_bound_set}
Suppose that the order of $x_t$ is preserved, and let $A\subseteq [0,1]$ be a set of positive measure. 
Then there exists a $\gamma\in A$ such that the set $\{\beta\in A: x_t(\beta) \leq 
x_t(\gamma)\hspace{.2cm} \forall t\}$ has a positive Lebesgue measure.
\end{lemma}
\begin{proof}
We first prove the existence of a $\gamma \in A$ such that $\{\beta\in A: x_0(\beta) \leq x_0(\gamma) \}$ has a positive Lebesgue measure, with the intention of showing later that this $\gamma$ can be used for establishing the statement of the lemma.

If there exists a $\gamma \in A$ such that the set $\{\beta\in A: x_0(\beta) < x_0(\gamma) \}$ (with a strict inequality) has a positive measure, then this $\gamma$ obviously satisfies our condition. 
Otherwise, it means that for every $\gamma \in A$, there holds
\begin{equation}\label{eq:zero_measure_below_gamma}
\Leb\{ \beta\in A: x_0(\beta) < x_0(\gamma)   \} = 0
\end{equation}
Let then $y^* := \inf \{y\in [0,1]: \Leb \{\beta\in N_\alpha^+, x_0(\beta) < y\}>0\}$. It follows from the definition of $y^*$ that the set of agents $\beta \in A$ having a value $x_0(\beta) < y^*$ has a zero measure. 
On the other hand, there is no agent $\gamma\in A$ for which 
$x_0(\gamma) >y^*$, for otherwise the definition of $y^*$ would imply that \eqref{eq:zero_measure_below_gamma} is not satisfied for that $\gamma$. Therefore, $x_0(\beta) = y^*$ must hold for every agent $\beta \in A$, except possibly those in a zero measure set (having a lower value). In particular, if we take any $\gamma$ outside that zero measure set, there holds $x_0(\gamma) = x_0(\beta)$ for almost every $\beta\in  A$. Since this set has a positive measure, we have thus shown the existence of a $\gamma \in A$ such that $\Leb \{\beta \in A: x_0(\beta) \leq x_0(\gamma) \} >0$ as in the first case treated.
 
Remember now that, since  the order of $x_t$ is preserved, there holds $x_t(\gamma) \geq x_t(\beta)$ for all $t$ if and only if $x_0(\gamma) \geq x_0(\beta)$. In particular, we have
$$
\{\beta\in A: x_0(\beta) \leq x_0(\gamma) \} = \{\beta\in A: x_t(\beta) \leq x_t(\gamma) \forall t \},
$$
so that our $\gamma$ satisfies the condition of this lemma.
\end{proof}

\smallskip

Let now $N$ be the set of agents $\alpha\in[0,1]$ for which $\lim_{t\to \infty}x_t(\alpha)$ does not exist, i.e. the set of the agents whose opinion does not converge, the measure of which we intend to prove is zero.

For every $\alpha \in N$, since $x_t(\alpha)$ remains in $[0,1]$ for all $t$, 
$\liminf _{t\to \infty} x_t(\alpha)$ and 
$\limsup _{t\to \infty} x_t(\alpha)$ are well defined, and the former is strictly smaller than the latter for otherwise $\lim_{t\to \infty}x_t(\alpha)$ would exist.
For every such $\alpha$, we let $I_\alpha := (\liminf _{t\to \infty} x_t(\alpha),$ $ \limsup _{t\to \infty} x_t(\alpha))$ be the open interval defined by the asymptotic lower and upper bounds on $x_t(\alpha)$.  The next lemma is instrumental to our proof and shows that the set $N_\alpha$ of agents $\beta\in N$ whose intervals $I_\beta$ intersect with $I_\alpha$ has a zero Lebesgue measure.

\smallskip

\begin{lemma}\label{lem:L(N_alpha)=0}
For every $\alpha\in N$, the set $N_\alpha= \{\beta: I_\beta \cap I_\alpha \neq \varnothing\}$ has a zero Lebesgue measure. 
\end{lemma}

\smallskip

\begin{proof}
Consider an $\alpha$. It is convenient to first treat the set $N_\alpha^+ = \{\beta\in N_\alpha, x_t(\beta) \geq x_t(\alpha)\hspace{.2cm} \forall t\}$. Suppose, to obtain a contradiction, that $N_\alpha^+$ has a positive Lebesgue measure. Lemma \ref{lem:lower_bound_set} 
implies then the existence of a $\gamma\in N^+_\alpha$ such that the set $\{\beta\in N_\alpha^+: x_t(\beta) \leq x_t(\gamma) \forall t \}$ has a positive measure, which directly implies that 
the set of agents whose opinion is between those of $\alpha$ and $\gamma$ also has a positive measure:
\begin{equation}\label{eq:positive_measure_between_a_gamma}
\Leb\{\beta: x_t(\alpha) \leq x_t(\beta) \leq x_t(\gamma), \forall t\}.
\end{equation}
Besides, the appartenance of $\gamma$ to $N_\alpha^+$ implies that the open intervals $I_\alpha = (\lim$ $\inf_{t\to\infty} x_t(\alpha),$ $\limsup_{t\to\infty} x_t(\alpha))$ and $I_\gamma = (\lim \inf_{t\to\infty} x_t(\gamma),$ $\limsup_{t\to\infty} x_t(\gamma))$ have a nonempty intersection and that  $x_t(\gamma)\geq x_t(\alpha)$ for all $t$, which is only possible if $\liminf_{t\to\infty} x_t(\gamma) < \limsup_{t\to\infty} x_t(\alpha)$. 

Take then $y_1,y_2\in [0,1]$ such that $\liminf_{t\to\infty} x_t(\gamma) < y_1<y_2<\limsup_{t\to\infty} x_t(\alpha)$, and let $f:\Re\to \Re$ be a function taking the value 1 for every $y\leq y_1$, $0$ for every $y\geq y_2$, and decreasing linearly between 1 and 0 for $y\in [y_1,y_2]$. This function is continuous, and it follows thus from the convergence of $x$ in distribution that $\lim_{t\to \infty} \int f(x_t(\beta)) d\beta$ exists (see Portmanteau's Theorem, for example in \cite{b86}). We will show that this creates a contradiction.

By definition of $y_1 > \liminf_{t\to\infty} x_t(\gamma)$, there exists a diverging sequence of times $\bar s_k$  at which $x_{\bar s_k}(\gamma) < y_1$. Consider such a time. Since $f$ is nonnegative, there holds
$$
\int f(x_{\bar s_k}(\beta))  d\beta \geq  \int_{\beta: x_{\bar s_k}(\beta) \leq y_1} f(x_{\bar s_k}(\beta)) = \Leb \{\beta: x_{\bar s_k}(\beta) \leq y_1\}\geq \Leb \{\beta: x_{\bar s_k}(\beta) \leq x_{\bar s_k}(\gamma)\},
$$
where the equality comes from the definition of $f$ and the second inequality comes from  $x_{\bar s_k}(\gamma) < y_1$.  Now remember that by the order preservation property, $x_{\bar s_k}(\beta) \leq x_{\bar s_k}(\gamma)$ if and only if $x_{t}(\beta) \leq x_{t}(\gamma)$ for all $t$. We have thus
\begin{equation}\label{eq:lbound_int_f}
\int f(x_{\bar s_k}(\beta))  d\beta \geq \Leb \{\beta: x_t(\beta) \leq x_t(\gamma) \hspace{.2cm}\forall t\}
\end{equation}

Similarly, there exists a diverging sequence of times $\underline s_k$  arbitrarily large times $t$  at which $x_{\underline s_k}(\alpha) > y_2$. Consider such a time.  Since $f(y)=0$ for $y>y_2$ and $f(y)\leq 1$ for all $y$, there holds
$$
\int f(x_{\underline s_k}(\beta))  d\beta =  \int_{\beta: x_{\underline s_k}(\beta) \leq y_2} f(x_{\underline s_k}(\beta)) \leq \Leb \{\beta: x_{\underline s_k}(\beta) \leq y_2\}, 
$$
The inequality $x_{\underline s_k}(\alpha) > y_2$, and the order preservation property imply then that 
\begin{equation}\label{eq:ubound_int_f}
\int f(x_{\underline s_k}(\beta))  d\beta \leq  \Leb \{\beta: x_{\underline s_k}(\beta) < x_{\underline s_k}(\alpha)\} = \Leb \{\beta: x_t(\beta) < x_t(\alpha) \hspace{.2cm}\forall t\}.
\end{equation} 
Now, since $\lim_{t\to\infty} \int f(x_t(\beta)) d\beta$ exists, and both sequences $\bar s_k$ and $\underline s_k$ diverges, the lower bound in \eqref{eq:lbound_int_f} and upper bound in \eqref{eq:ubound_int_f} must be equal:
$$
\Leb \{\beta: x_t(\beta) \leq x_t(\gamma) \hspace{.2cm}\forall t\} =  \Leb \{\beta: x_t(\beta) < x_t(\alpha) \hspace{.2cm}\forall t\}.
$$
Remembering that, due to the order preservation, either $x_t(\beta) < x_t(\alpha)$ for all $t$, or $x_t(\beta) \geq x_t(\alpha)$ for all $t$, the equality above implies that  $\Leb \{\beta: x_t(\alpha) \leq x_t(\beta) \leq x_t(\gamma)\} = 0$, in contradiction with \eqref{eq:positive_measure_between_a_gamma}. We must thus have $\Leb(N_\alpha^+)= 0$.

A symmetric argument shows that the set $N_\alpha^- = \{\beta\in N_\alpha, x_t(\beta) \leq x_t(\alpha)\hspace{.2cm} \forall t\}$. Using again the fact that either $x_t(\beta) < x_t(\alpha)$ for all $t$, or $x_t(\beta) \geq x_t(\alpha)$ for all $t$, we obtain $N_\alpha = N_\alpha^+\cup N_\alpha^-$, so that  $\Leb(N_\alpha)=0$.\end{proof}
\smallskip

We can now prove Lemma \ref{lem:order+weak->a.e.}

\smallskip
\begin{proof}
Remember that the open interval $I_\alpha := ( \liminf _{t\to \infty} x_t(\alpha), \limsup _{t\to \infty} x_\alpha(t))$ has a positive length for every $\alpha \in N$ since $x_t(\alpha)$ does not converge for any $\alpha \in N$. For every $\lambda>0$, let $N^\lambda$ be the set of agents $\alpha \in N$ for which $I_\alpha$ has a length larger than $\lambda$. Since $N$ can be written as a countable union of $N^\lambda$ (just take $\bigcup_m N^{1/m}$ for example), it has a positive measure if and only if $N^\lambda$ has a positive measure for at least one $\lambda>0$. We show that this is impossible.

Suppose indeed that $\Leb(N^\lambda )>0$, and chose an $\alpha_0\in N^\lambda$. Let then $N_1^\lambda:=\{\beta\in N^\lambda:I_\beta \cap I_{\alpha_0} = \varnothing  \}$. It follows 
from Lemma \ref{lem:L(N_alpha)=0} 
that $\Leb(N_1^\lambda) = \Leb(N^\lambda)>0$. On the 
other hand, since $I_{\alpha_0}$ has a length at least $\lambda$, and since $I_\beta \subseteq [0,1]$ for every $\beta\in N$, there holds
$
\bigcup_{\beta \in N_1^\lambda} I_\beta \subseteq [0,1]\setminus I_{\alpha_0},
$
so that the measure of $\bigcup_{\beta \in N_1^\lambda} I_\beta$ is at most $1-\lambda$. Since $N_1^\lambda$ has a positive measure, it is non-empty, and we can chose an $\alpha_1\in N_1^{\lambda}$. Let then $N_2^\lambda := \{\beta\in N_1^\lambda:I_\beta \cap I_{\alpha_1} = \varnothing  \}$. Again, it follows from Lemma \ref{lem:L(N_alpha)=0} 
that $\Leb(N_2^\lambda) = \Leb(N_1^\lambda) = \Leb(N^\lambda)>0$, and from the definitions of $N_1^\lambda,N_2^\lambda$ that $
\bigcup_{\beta \in N_2^\lambda} I_\beta \subseteq [0,1]\setminus (I_{\alpha_0}\cup I_{\alpha_1}).
$
Moreover, since $I_{\alpha_0}$ and $I_{\alpha_1}$ are disjoint by construction (remember indeed that $\alpha_1\in N_1^\lambda$), the Lebesgue measure of  $\bigcup_{\beta \in N_2^\lambda}I_\beta$ is at most $1-2\lambda$. We can then continue defining sets $N_k^\lambda$ recursively, keeping $\Leb(N_k^\lambda) = \Leb(N^\lambda)>0$ while having a Lebesgue measure at most $1-k\lambda$ for the set $\bigcup_{\beta \in N_k^\lambda}$. This is however impossible since that measure would be negative for $k>1/\lambda$. Therefore $N^\lambda$ must have a measure zero for every $\lambda >0$, and so has thus the set $N$ of agents whose opinion does not converge, which achieves the proof.\end{proof}

\end{document}